\documentclass[11pt]{amsart}

\usepackage{graphicx,geometry,esint,amssymb,ulem}
\usepackage[colorlinks=true,urlcolor=blue,
citecolor=red,linkcolor=blue,linktocpage,pdfpagelabels,bookmarksnumbered,bookmarksopen]{hyperref}
\usepackage[hyperpageref]{backref}
\usepackage{verbatim,color}

\usepackage{tikz}
\usepackage{pgfplots}

\newtheorem{lm}{Lemma}[section]
\newtheorem{prop}[lm]{Proposition}
\newtheorem{corollary}[lm]{Corollary}
\newtheorem{teo}[lm]{Theorem}

\theoremstyle{definition}

\newtheorem{defi}[lm]{Definition}
\newtheorem{oss}[lm]{Remark}
\newtheorem*{ack}{Acknowledgments}

\title[Schr\"odinger operators and Lane-Emden densities]{Schr\"odinger operators\\ with negative potentials\\ and Lane-Emden densities}

\author[Brasco]{Lorenzo Brasco}
\author[Franzina]{Giovanni Franzina}
\author[Ruffini]{Berardo Ruffini}

\address[L.\ Brasco]{Dipartimento di Matematica e Informatica
	\newline\indent
	Universit\`a degli Studi di Ferrara
	\newline\indent
	Via Machiavelli 35, 44121 Ferrara, Italy}
\address{{\it and }
	Aix Marseille Univ, CNRS, Centrale Marseille, I2M, Marseille, France
	\newline\indent
	39 Rue Fr\'ed\'eric Joliot Curie, 13453 Marseille}
\email{lorenzo.brasco@unife.it}

\address[G. Franzina]{Dipartimento di Matematica ``G. Castelnuovo''\newline\indent 
	Sapienza Universit\`a di Roma
	\newline\indent
	Piazzale Aldo Moro 5, 00185 Roma, Italy}
\email{franzina@mat.uniroma1.it}

\address[B.\ Ruffini]{Institut Montpelli\'erain Alexander Grothendieck,\newline\indent CNRS, Univ. Montpellier,
34095 Montpellier, Cedex 5, France}
\email{berardo.ruffini@umontpellier.fr}


\numberwithin{equation}{section}

\def\XXint#1#2#3{{\setbox0=\hbox{$#1{#2#3}{\int}$} \vcenter{\vspace{-1pt}\hbox{$#2#3$}}\kern-.5\wd0}}

\begin{document}

\begin{abstract}
We consider the Schr\"odinger operator $-\Delta+V$ for negative potentials $V$, on open sets with positive first eigenvalue of the Dirichlet-Laplacian. We show that the spectrum of $-\Delta+V$ is positive, provided that $V$ is greater than a negative multiple of the logarithmic gradient of the solution to the Lane-Emden equation $-\Delta u=u^{q-1}$ (for some $1\le q< 2$). In this case, the ground state energy of $-\Delta+V$ is greater than the first eigenvalue of the Dirichlet-Laplacian, up to an explicit multiplicative factor. This is achieved by means of suitable Hardy-type inequalities, that we prove in this paper.
\end{abstract}

\maketitle

\begin{center}
\begin{minipage}{10cm}
\small
\tableofcontents
\end{minipage}
\end{center}

\section{Introduction}

\subsection{Foreword}

Let $V\in L^2_{\rm loc}(\mathbb{R}^N)$ be a real-valued potential such that $V\le 0$ and let us consider the {\it Schr\"odinger operator} $\mathcal{H}_V:=-\Delta+V$, acting on the domain
\[
\mathfrak{D}(\mathcal{H}_V):=H^2(\mathbb{R}^N)\cap\{u\in L^2(\mathbb{R}^N)\, :\, V\,u\in L^2(\mathbb{R}^N)\}.
\]
Observe that the hypothesis $V\in L^2_{\rm loc}(\mathbb{R}^N)$ entails the inclusion
\[
C^\infty_0(\mathbb{R}^N)\subset \mathfrak{D}(\mathcal{H}_{V}),
\]
thus $\mathfrak{D}(\mathcal{H}_V)$ is dense in $L^2(\mathbb{R}^N)$.
The operator $\mathcal{H}_V:\mathfrak{D}(\mathcal{H}_V)\to L^2(\mathbb{R}^N)$ is symmetric and self-adjoint as well, thanks to the fact that $V$ is real-valued (see \cite[Example p. 68]{Te}). The {\it spectrum} of $\mathcal{H}_V$ is the set
\[
\sigma(\mathcal{H}_V)=\mathbb{R}\setminus \rho(\mathcal{H}_V),
\]
where $\rho(\mathcal{H}_V)$ is the {\it resolvent set of} $\mathcal{H}_V$, defined as the collection of real numbers $\lambda$ such that $\mathcal{H}_V-\lambda$ is bijective
and its inverse is a bounded linear operator.
\par
A distinguished subset of $\sigma(\mathcal{H}_V)$ is given by the collection of those $\lambda$ such that the kernel of $\mathcal{H}_V-\lambda $ is nontrivial. In this case, the {\it stationary  Schr\"odinger equation}
\begin{equation}
\label{schrodinger}
\mathcal{H}_V\, u = \lambda\, u,
\end{equation}
admits a nontrivial solution $u\in \mathfrak{D}(\mathcal{H}_V)$.
Whenever this happens, $\lambda$ is called an {\it eigenvalue} of the Schr\"odinger operator. Correspondingly, the solution is said to be an {\it eigenfunction} corresponding to $\lambda$.
\par
The operator $\mathcal{H}_V$ comes with the associated quadratic form
\[
\varphi\mapsto \mathcal{Q}_V(\varphi)= \int_{\mathbb R^N}|\nabla \varphi|^2\,dx+\int_{\mathbb R^N} V\,\varphi^2\,dx,\qquad \varphi\in \mathfrak{D}(\mathcal{H}_V).
\]
From classical Spectral Theory, we have (see \cite[Theorem 2.20]{Te})
\begin{equation}
\label{bottom}
\inf \sigma(\mathcal{H}_V)=\inf_{\varphi\in \mathfrak{D}(\mathcal{H}_V)}\left\{\mathcal{Q}_V(\varphi)\, :\, \int_{\mathbb{R}^N} \varphi^2\,dx=1\right\}.
\end{equation}
We call such a value {\it ground state energy of $\mathcal{H}_V$}. 
\par
This quantity is important in classical Quantum Mechanics, since it is the lowest energy that a particle in $\mathbb R^N$ interacting with the force field generated by the potential $V$ can attain (and which will eventually attain by emitting energy). From a mathematical point of view, we observe that the stationary Schr\"odinger equation \eqref{schrodinger} is precisely the Euler-Lagrange equation of the minimization problem appearing in \eqref{bottom}.
\par
An issue of main interest is providing a lower bound on the ground state energy
(and thus on the spectrum) of $\mathcal{H}_V$. 
\vskip.2cm
It is well-known that when $V\equiv 0$, then $\inf \sigma(\mathcal{H}_V)=0$. On the other hand, if we take $V\le 0$,
the kinetic energy $\int_{\mathbb R^N}|\nabla \varphi|^2\,dx$  and the potential energy $\int_{\mathbb R^N} V\,\varphi^2\,dx$
are in competition in the quadratic form $\mathcal{Q}_V$ and one could expect that 
\[
\inf\sigma(\mathcal{H}_V)<0.
\] 
Actually, this depends on the potential $V$. For example, by recalling the Hardy inequality  on $\mathbb R^N$ (for $N\ge3$)
\[
\left(\frac{N-2}{2}\right)^2\,\int_{\mathbb R^N} \frac{\varphi^2}{|x|^2}\,dx\le  \int_{\mathbb R^N}|\nabla \varphi|^2\,dx,\qquad \varphi\in C^\infty_0(\mathbb{R}^N\setminus\{0\}),
\]
we get that if the potential $V$ is such that
\[
0\ge V\ge -\left(\frac{N-2}{2}\right)^2\,\frac{1}{|x|^2},
\]
then the spectrum of $\mathcal{H}_V$ is still non-negative. This is an example of how Hardy-type inequalities can be exploited  in order to identify classes of negative potentials with non-negative spectrum.

\subsection{Aim of the paper} In this paper we deal with a {\it confined} version of this problem. More precisely, we turn our attention to prescribed open sets $\Omega\subset \mathbb{R}^N$. We fix a potential $V\in L^2_{\rm loc}(\Omega)$ such that $V\le 0$ and consider the localized Schr\"odinger operator with homogeneous boundary conditions $\mathcal{H}_{\Omega,V}=-\Delta+V$,
this time acting on the domain
\begin{equation}
\label{dominio}
\mathfrak{D}(\mathcal{H}_{\Omega,V}):=H^2(\Omega)\cap H^1_0(\Omega)\cap\{u\in L^2(\Omega)\, :\, V\,u\in L^2(\Omega)\}.
\end{equation}
Here $H^1_0(\Omega)$ is the closure of $C^\infty_0(\Omega)$ in the Sobolev space $H^1(\Omega)$. This is still a symmetric  and self-adjoint operator $\mathcal{H}_V:\mathfrak{D}(\mathcal{H}_{\Omega,V})\to L^2(\Omega)$, with real spectrum $\sigma(\mathcal{H}_{\Omega,V})$.
Observe that the hypothesis $V\in L^2_{\rm loc}(\Omega)$ entails as before the inclusion
\[
C^\infty_0(\Omega)\subset \mathfrak{D}(\mathcal{H}_{\Omega,V}),
\]
thus the operator is densely defined.
We define the associated quadratic form
\[
\mathcal{Q}_{\Omega,V}(\varphi)=\int_{\Omega} |\nabla \varphi|^2\,dx+\int_\Omega V\,\varphi^2\,dx,\qquad \varphi\in \mathfrak{D}(\mathcal{H}_{\Omega,V}).
\]
The stationary equation \eqref{schrodinger} now reads
\begin{equation}\label{schrodinger2}
\left\{\begin{array}{rcll}
\mathcal{H}_{\Omega,V}\, u&=&\lambda\, u & \mbox{ in } \Omega,\\
u&=&0,&\mbox{ in }\mathbb R^N\setminus\Omega.
\end{array} 
\right.
\end{equation}
Equation \eqref{schrodinger2} can be formally considered as a peculiar form of \eqref{schrodinger}, where the potential $V$ has the {\it trapping property} $V=+\infty$ in $\mathbb{R}^N\setminus \Omega$. This models the physical situation where the particle is ``trapped'' in the confining region $\Omega$.
\par
The issue we tackle is the following
\vskip.2cm
\centerline{``{\it find explicit pointwise bounds on the potential $V$}}
\centerline{{\it assuring that the ground state energy of $\mathcal{H}_{\Omega,V}$ stays positive\,}''}
\vskip.2cm
In the vein of the example discussed above using Hardy's inequality in the entire space,  we will approach this problem by proving localized Hardy-type inequalities with suitable weights.
A typical instance of these inequalities
occurs when we limit ourselves to consider functions supported in a proper open subset $\Omega\subset \mathbb R^N$
and
we use the distance $d_\Omega(x):={\rm dist}(x,\partial \Omega)$ as a weight. In other words, one has
\[
\frac{1}{C}\,\int_{\Omega} \frac{\varphi^2}{d_\Omega^2}\,dx\le  \int_{\Omega}|\nabla \varphi|^2\,dx,\qquad \varphi\in C^\infty_0(\Omega).
\]
However, the existence of such a constant $C>0$ typically requires some conditions on the geometry of the set $\Omega$ or on the regularity of its boundary, see \cite{OK}.
In this paper on the contrary, we will prove alternative Hardy-type inequalities, with weights depending on solutions of peculiar elliptic partial differential equations.
\par
Roughly speaking, we will consider the solution $w_{q,\Omega}$ to the {\it Lane-Emden equation}\footnote{The terminology comes from astrophysics, where the Lane-Emden equation is
\[
\frac{1}{\varrho^2}\,\frac{d}{d\varrho}\left(\varrho^2\,\frac{du}{d\varrho}\right)+u^\gamma=0,
\]
for a radially symmetric function $u:\mathbb{R}^3\to\mathbb{R}$. The positive number $\gamma$ is usually called {\it polytropic index}. Observe that for a radial function $u$ defined in $\mathbb{R}^3$, this is equivalent to
\[
-\Delta u=u^\gamma.
\]
Though our paper is not concerned with astrophysics, we found useful to give a name to the equation and its solution.} with $1\le q<2$
\begin{equation}
\label{LEintro}
\left\{\begin{array}{rcll}
-\Delta u&=&u^{q-1}&\mbox{ in } \Omega,\\
u&=&0, &\mbox{ in } \mathbb R^N\setminus\Omega,\\
u&>&0, & \mbox{ in } \Omega,
\end{array}
\right.
\end{equation}
prove a Hardy inequality with weight depending on $w_{q,\Omega}$ and show that the condition
\[
0\ge V\gtrsim -\left|\frac{\nabla w_{q,\Omega}}{w_{q,\Omega}}\right|^2,\qquad \mbox{ a.\,e. in }\Omega,
\] 
leads to positivity of the spectrum of the Schr\"odinger operator $\mathcal{H}_{\Omega,V}$. 
\par
The function $w_{q,\Omega}$ will be called the {\it Lane-Emden $q-$density of} $\Omega$, we refer to Definitions \ref{defi:boundeddensity} and \ref{defi:generaldensity} below.

\subsection{Main results}

Let us now try to be more precise about our results. We first need to fix some definitions.
For $\gamma\ge 1$, we denote
\begin{equation}
\label{lambda2g}
\lambda_{2,\gamma}(\Omega)=\inf_{\varphi\in C^\infty_0(\Omega)} \left\{\int_\Omega |\nabla \varphi|^2\,dx\, :\, \|\varphi\|_{L^\gamma(\Omega)}=1\right\}.
\end{equation}
Henceforth we shall often work with the following class of sets.
\begin{defi}
We say that $\Omega\subset\mathbb{R}^N$ is an {\it open set with positive spectrum} if it is open and 
\begin{equation}
\label{lambda1}
\lambda_1(\Omega):=\lambda_{2,2}(\Omega)=\inf_{\varphi\in C^\infty_0(\Omega)} \left\{\int_\Omega |\nabla \varphi|^2\,dx\, :\, \|\varphi\|_{L^2(\Omega)}=1\right\}>0.
\end{equation}
\end{defi}
The main result of the paper is the following lower bound on the ground state energy of $\mathcal{H}_{\Omega,V}$. We refer to Theorem \ref{teo:mainthm} and Corollary \ref{coro:teo:mainthm} for its proof.
\begin{teo}
\label{teo:mainthmintro}
Let $\Omega\subset\mathbb{R}^N$ be an open set with positive spectrum,
and let $V\in L^2_{\rm loc}(\Omega)$. For  an exponent $1\le q<2$, we assume that
\[
0\ge V\ge -\frac{1}{4}\,\left|\frac{\nabla w_{q,\Omega}}{w_{q,\Omega}}\right|^2,\qquad \mbox{a.\,e. in }\Omega.
\]
Then the spectrum $\sigma(\mathcal{H}_{\Omega,V})$ of $\mathcal{H}_{\Omega,V}$ is positive and we have that
\[
\inf \sigma(\mathcal{H}_{\Omega,V})=\inf_{\varphi\in C^\infty_0(\Omega)}\left\{\mathcal{Q}_{\Omega,V}(\varphi)\, :\, \int_{\Omega} \varphi^2\,dx=1\right\}\ge \frac{1}{C}\,\lambda_1(\Omega),
\]
where $C=C(N,q)>0$ is an explicit constant.
\end{teo}
We point out that due to the quantitative estimate of the previous result, the condition on $V$ can be slightly relaxed and still we can have positivity of the spectrum. We refer to Remark \ref{oss:referee} below for more details on this point.
\par
As stated above, the main tool we use to prove Theorem \ref{teo:mainthmintro} is an Hardy-type inequality, in which a weight involving the solution $w_{q,\Omega}$ of the Lane-Emden equation \eqref{LEintro} enters. This is the content of the next result. For questions related to optimal choices of weights in Hardy-type inequalities, see~\cite{devfrapin} and the references therein.
\begin{teo}[Hardy-Lane-Emden inequality]
\label{thm:freeHLEintro}
Let $1\le q<2$ and let $\Omega\subset\mathbb{R}^N$ be an open set with positive spectrum. Then for every $\varphi\in C^\infty_0(\Omega)$ and $\delta>0$ we have that
\[
\frac{1}{\delta}\,\left(1-\frac{1}{\delta}\right)\,\int_\Omega \left|\frac{\nabla w_{q,\Omega}}{w_{q,\Omega}}\right|^2\,\varphi^2\,dx+\frac{1}{\delta}\,\int_\Omega\frac{\varphi^2}{w^{2-q}_{q,\Omega}}\,dx\le \int_\Omega |\nabla \varphi|^2\,dx.
\]
\end{teo}
We refer to Remark \ref{oss:proofhardy} for some comments about the proof of this result.
\subsection{Plan of the paper}
The paper is organized as follows:
in Section \ref{sec:2}, we define the Lane-Emden $q-$density of a set $\Omega\subset\mathbb{R}^N$, first under the assumption that $\Omega$ is bounded and then for a general open set. Then in Section \ref{sec:3} we prove the Hardy-Lane-Emden inequality of Theorem \ref{thm:freeHLEintro} for bounded open sets. 
\par
In Section \ref{sec:4} we show how the summability properties of the Lane-Emden densities are equivalent to the embedding of $\mathcal{D}^{1,2}_0(\Omega)$ into suitable Lebesgue spaces. This part generalizes some results contained in the recent paper \cite{braruf}, by replacing the torsion function with any Lane-Emden $q-$density. Though this section may appear unrelated to ground state energy estimates for $\mathcal{H}_{\Omega,V}$, some of its outcomes are used to extend (in Section \ref{sec:5}) the Hardy-Lane-Emden inequality to open sets with positive spectrum.
\par
The proof of Theorem \ref{teo:mainthmintro} is then contained in Section \ref{sec:6}, while Section \ref{sec:7} contains some applications of our main result to some particular geometries (a ball, an infinite slab and a rectilinear wave-guide with circular cross-section).
\par
We conclude the paper with an Appendix, containing a local $L^\infty$ estimate for subsolutions of the Lane-Emden equation, which is necessary in order to get the explicit lower bound on the ground state energy of $\mathcal{H}_{\Omega,V}$.

\begin{ack}
The first author would like to thank Douglas Lundholm for a discussion on Hardy inequalities and the so-called {\it Ground State Representation} in February 2017, during a visit to the  Department of Mathematics of KTH (Stockholm). He also wishes to thank Erik Lindgren for the kind invitation. Remark \ref{oss:superhomo} comes from an informal discussion with Guido De Philippis in December 2015, we wish to thank him.
\par
The authors are members of the Gruppo Nazionale per l'Analisi Matematica, la Pro\-ba\-bi\-li\-t\`a
e le loro Applicazioni (GNAMPA) of the Istituto Nazionale di Alta Matematica (INdAM).
\end{ack}

\section{Preliminaries}
\label{sec:2}

\subsection{Notation} Let $\Omega\subset\mathbb{R}^N$ be an open set and define the norm on $C^\infty_0(\Omega)$
\[
\|\varphi\|_{\mathcal{D}^{1,2}_0(\Omega)}=\left(\int_\Omega |\nabla \varphi|^2\,dx\right)^\frac{1}{2},\qquad \varphi\in C^\infty_0(\Omega).
\]
We consider the {\it homogeneous Sobolev space} $\mathcal{D}^{1,2}_0(\Omega)$, obtained as the
completion of $C^\infty_0(\Omega)$ with respect to the norm $\|\cdot\|_{\mathcal{D}^{1,2}_0(\Omega)}$. For $N\ge 3$ this is always a functional space, thanks to Sobolev inequality but in dimension $N=1$ or $N=2$, this may fail to be even a space of distributions if $\Omega$ is ``too big'', see for example \cite[Remark 4.1]{DL}.
\begin{oss}
\label{oss:uguali}
For an open set with positive spectrum $\Omega$, we have automatically continuity of the embedding $\mathcal{D}^{1,2}_0(\Omega)\hookrightarrow L^2(\Omega)$. Thus in this case $\mathcal{D}^{1,2}_0(\Omega)$ is a functional space. Moreover, we have that
\[
\mathcal{D}^{1,2}_0(\Omega)=H^1_0(\Omega),
\] 
thanks to the fact that in this case
\[
\left(\int_\Omega |\nabla \varphi|^2\,dx\right)^\frac{1}{2}\qquad \mbox{ and } \qquad \left(\int_\Omega |\nabla \varphi|^2\,dx\right)^\frac{1}{2}+\left(\int_\Omega \varphi^2\,dx\right)^\frac{1}{2},
\]
are equivalent norms on $C^\infty_0(\Omega)$.
\end{oss}

\subsection{Lane-Emden densities: bounded sets}
We start with the following auxiliary result.
\begin{lm}
Let $\Omega\subset\mathbb{R}^N$ be an open bounded set. For $1\le q<2$, the variational problem
\begin{equation}
\label{LEmin}
\min_{\varphi\in \mathcal{D}^{1,2}_0(\Omega)}\left\{\frac{1}{2}\,\int_\Omega |\nabla \varphi|^2\,dx-\frac{1}{q}\,\int_\Omega \varphi^q\,dx\, :\, \varphi\ge 0 \mbox{ a.\,e. in }\Omega\right\},
\end{equation}
admits a unique solution.
\end{lm}
\begin{proof}
Since the absolute value of every minimizer of the functional
\[
\varphi\mapsto \frac{1}{2}\,\int_\Omega |\nabla \varphi|^2\,dx-\frac{1}{q}\,\int_\Omega |\varphi|^q\,dx, 
\]
is also a minimizer of \eqref{LEmin}, problem \eqref{LEmin} is equivalent to
\[
\min_{\varphi\in \mathcal{D}^{1,2}_0(\Omega)}\left\{\frac{1}{2}\,\int_\Omega |\nabla \varphi|^2\,dx-\frac{1}{q}\,\int_\Omega |\varphi|^q\,dx\right\}.
\]
The existence of a solution follows then by the Direct Methods in the Calculus of Variations, since the embedding $\mathcal{D}^{1,2}_0(\Omega)\hookrightarrow L^q(\Omega)$ is compact and $\mathcal{D}^{1,2}_0(\Omega)$ is weakly closed.
\vskip.2cm\noindent
As for uniqueness, we first suppose that $\Omega$ is connected. We observe that for $q=1$ problem \eqref{LEmin} is strictly convex, thus the solution is unique. For $1<q<2$, we can use a trick by Brezis and Oswald based on the so-called {\it Picone's inequality}, see \cite[Theorem 1]{BO}. We reproduce their argument here for completeness. We first observe that a minimizer is a positive solution of the Lane-Emden equation
\begin{equation}
\label{LE}
-\Delta u=u^{q-1},\qquad \mbox{ in }\Omega,
\end{equation}
with homogeneous Dirichlet boundary conditions. More precisely, for every $\varphi\in \mathcal D_0^{1,2}(\Omega)$ it holds
\begin{equation}\label{wLE}
\int_{\Omega}\langle\nabla u,\nabla\varphi\rangle\,dx=\int_\Omega u^{q-1}\,\varphi\,dx.
\end{equation}
We now suppose that \eqref{LEmin} admits two minimizers $u_1,u_2\in \mathcal{D}^{1,2}_0(\Omega)$. By the minimum principle for superharmonic functions, $u_1>0$ and $u_2>0$ on $\Omega$. Moreover, by standard Elliptic Regularity, $u_1,u_2\in L^\infty(\Omega)$. We fix $\varepsilon>0$, then we test equation \eqref{wLE} for $u_1$ with
\[
\varphi=\frac{u^2_2}{u_1+\varepsilon}-u_1,
\]
and equation \eqref{wLE} for $u_2$ with
\[
\varphi=\frac{u^2_1}{u_2+\varepsilon}-u_2.
\]
Summing up, we get that
\[
\begin{split}
\int_\Omega \left\langle \nabla u_1,\nabla\left(\frac{u_2^2}{u_1+\varepsilon}\right)\right\rangle\,dx&-\int_\Omega |\nabla u_1|^2\,dx+\int_\Omega \left\langle \nabla u_2,\nabla\left(\frac{u_1^2}{u_2+\varepsilon}\right)\right\rangle\,dx-\int_\Omega |\nabla u_2|^2\,dx\\
&=\int_\Omega \frac{u_1^{q-1}}{u_1+\varepsilon}\,u_2^2\,dx-\int_\Omega u_1^q\,dx+\int_\Omega \frac{u_2^{q-1}}{u_2+\varepsilon}\,u_1^2\,dx-\int_\Omega u_2^q\,dx.
\end{split}
\]
We now recall that
\begin{equation}
\label{pitone}
\left\langle \nabla u,\nabla \left(\frac{v^2}{u}\right)\right\rangle\le |\nabla v|^2,
\end{equation}
for $v$ and $u>0$ differentiable. This is precisely Picone's inequality, see for example \cite{brafra}. By observing that $\nabla u_i=\nabla (u_i+\varepsilon)$ and using \eqref{pitone} in the identity above, we conclude that
\[
\int_\Omega \frac{u_1^{q-1}}{u_1+\varepsilon}\,u_2^2\,dx-\int_\Omega u_1^q\,dx+\int_\Omega \frac{u_2^{q-1}}{u_2+\varepsilon}\,u_1^2\,dx-\int_\Omega u_2^q\,dx\le 0.
\] 
We now take the limit as $\varepsilon$ goes to $0$. By Fatou's Lemma, we obtain that
\[
\int_\Omega u_1^{q-2}\,u_2^2\,dx-\int_\Omega u_1^q\,dx+\int_\Omega u_2^{q-2}\,u_1^2\,dx-\int_\Omega u_2^q\,dx\le 0.
\]
The previous terms can be recast into inequality
\[
\int_\Omega (u_2^2-u_1^2)\,(u_2^{q-2}-u_1^{q-2})\,dx\ge 0.
\]
By using the fact the function $t\mapsto t^{q-2}$ is monotone decreasing, we get that $u_1=u_2$ as desired.
\vskip.2cm\noindent
Finally, if $\Omega$ is not connected, it is sufficient to observe that a solution of \eqref{LEmin} must minimize the same functional on every connected component, due to the locality of the functional; since the solution is unique on every connected component, we get the conclusion in this case as well.
\end{proof}
\begin{oss}[About uniqueness]
Uniqueness of the solution to \eqref{LEmin} can also be inferred directly at the level of the minimization problem. It is sufficient to observe that the functional to be minimized is convex along curves of the form
\[
\gamma_t=\Big((1-t)\,\varphi_0^q+t\,\varphi_1^q\Big)^\frac{1}{q},\qquad t\in[0,1],\quad \varphi_0,\varphi_1\in \mathcal{D}^{1,2}_0(\Omega) \mbox{ positive},
\]
see \cite[Proposition 2.6]{brafra}. Then one can reproduce the uniqueness proof of \cite{BK}.
For a different proof of the uniqueness for \eqref{LE}, we also refer to \cite[Corollary 4.2]{FraLa}.
\end{oss}
\begin{oss}
It is useful to keep in mind that if $u\in \mathcal{D}^{1,2}_0(\Omega)$ solves equation
\[
-\Delta u=t\,u^{q-1},\qquad \mbox{ in }\Omega,
\]
for some $t>0$, then the new function
\[
v_t=t^\frac{1}{q-2}\,u,
\]
solves \eqref{LE}.
\end{oss}

\begin{defi}
\label{defi:boundeddensity}
Let $\Omega\subset\mathbb{R}^N$ be an open bounded set. For $1\le q<2$, we define the {\it Lane-Emden $q-$density of} $\Omega$ as the unique solution of \eqref{LEmin}.
We denote such a solution by $w_{q,\Omega}$. In the case $q=1$, we simply write $w_\Omega$ and call it {\it torsion function of} $\Omega$.
\end{defi}
The variational problem defining $w_{q,\Omega}$ is related to the optimal Poincar\'e constant $\lambda_{2,q}(\Omega)$ defined in \eqref{lambda2g}. This is the content of the next result, that we record for completeness. We omit the proof 
since it is based on a straightforward scaling argument.
\begin{lm}
Let $1\le q<2$ and let $\Omega\subset\mathbb{R}$ be an open bounded set. Then we have
\begin{equation}
\label{qLE}
\min_{\varphi\in \mathcal{D}^{1,2}_0(\Omega)}\left\{\frac{1}{2}\,\int_\Omega |\nabla \varphi|^2\,dx-\frac{1}{q}\,\int_\Omega \varphi^q\,dx\, :\, \varphi\ge 0 \mbox{ in }\Omega\right\}=\frac{q-2}{2\,q}\,\left(\frac{1}{\lambda_{2,q}(\Omega)}\right)^\frac{q}{2-q}
\end{equation}
and
\begin{equation}
\label{LEtorsion}
\left(\int_\Omega |w_{q,\Omega}|^q\,dx\right)^\frac{2-q}{q}=\frac{1}{\lambda_{2,q}(\Omega)}.
\end{equation}
\end{lm}
\subsection{Lane-Emden densities: general sets}

We now want to define the Lane-Emden densities for a general open set, where the variational problem
\[
\inf_{\varphi\in \mathcal{D}^{1,2}_0(\Omega)}\left\{\frac{1}{2}\,\int_\Omega |\nabla \varphi|^2\,dx-\frac{1}{q}\,\int_\Omega \varphi^q\,dx\, :\, \varphi\ge 0 \mbox{ in }\Omega\right\},
\]
may fail to admit a solution. 
\vskip.2cm\noindent
We start with a sort of comparison principle for Lane-Emden densities.
\begin{lm}\label{lm:comparison}
Let $1\le q<2$ and let $\Omega_1\subset\Omega_2\subset\mathbb{R}^N$ be two open bounded sets. Then we have
\[
w_{q,\Omega_1}\le w_{q,\Omega_2}.
\]
\end{lm}
\begin{proof}
We test the minimality of $w_{q,\Omega_1}$ against $\varphi=\min\{w_{q,\Omega_1},w_{q,\Omega_2}\}$. After some simple manipulations, this gives
\[
\begin{split}
\frac{1}{2}\,\int_{\{w_{q,\Omega_2}<w_{q,\Omega_1}\}}& |\nabla w_{q,\Omega_2}|^2\,dx-\frac{1}{q}\,\int_{\{w_{q,\Omega_2}<w_{q,\Omega_1}\}} w_{q,\Omega_2}^q\,dx\\
&\ge \frac{1}{2}\,\int_{\{w_{q,\Omega_2}<w_{q,\Omega_1}\}} |\nabla w_{q,\Omega_1}|^2\,dx-\frac{1}{q}\,\int_{\{w_{q,\Omega_2}<w_{q,\Omega_1}\}} w_{q,\Omega_1}^q\,dx.
\end{split}
\]
We now add on both sides the term
\[
\frac{1}{2}\,\int_{\{w_{q,\Omega_2}>w_{q,\Omega_1}\}} |\nabla w_{q,\Omega_2}|^2\,dx-\frac{1}{q}\,\int_{\{w_{q,\Omega_2}>w_{q,\Omega_1}\}} w_{q,\Omega_2}^q\,dx,
\]
thus if set $U=\max\{w_{q,\Omega_1},w_{q,\Omega_2}\}$, we get that
\[
\frac{1}{2}\,\int_{\Omega_2} |\nabla w_{q,\Omega_2}|^2\,dx-\frac{1}{q}\,\int_{\Omega_2} w_{q,\Omega_2}^q\,dx\ge \frac{1}{2}\,\int_{\Omega_2} |\nabla U|^2\,dx-\frac{1}{q}\,\int_{\Omega_2} U^q\,dx.
\]
By uniqueness of the minimizer $w_{q,\Omega_2}$, this gives $U=w_{q,\Omega_2}$. By recalling the definition of $U$, this in turn yields the desired conclusion.
\end{proof}
Thanks to the previous property, we can define the Lane-Emden density for every open set. In what follows, we set
\[
\Omega_R=\Omega\cap B_R(0),\qquad R>0,
\]
where $B_R(0)$ is the $N-$dimensional open ball, with radius $R$ and centered at the origin.
\begin{defi}
\label{defi:generaldensity}
Let $\Omega\subset \mathbb{R}^N$ be an open set. For $1\le q<2$ we define 
\[
w_{q,\Omega}=\lim_{R\to +\infty} w_{q,\Omega_R}.
\]
\end{defi}
We observe that this definition is well-posed, since each $w_{q,\Omega_R}\in\mathcal{D}^{1,2}_0(\Omega_R)$ exists thanks to the boundedness of $\Omega_R$ and the function
\[
R\mapsto w_{q,\Omega_R}(x),
\]
is monotone, thanks to Lemma \ref{lm:comparison}.

\begin{oss}[Consistency]
When $\Omega\subset\mathbb{R}^N$ is an open bounded set or, more generally, is such that the embedding $\mathcal{D}^{1,2}_0(\Omega)\hookrightarrow L^q(\Omega)$ is compact, then the definition of $w_{q,\Omega}$ above coincides with the variational one. For $q=1$ this is proved in \cite[Lemma 2.4]{braruf}, the other cases can be treated in exactly the same way. We skip the details.
\end{oss}

\section{Hardy-Lane-Emden inequalities}
\label{sec:3}
The following theorem, which is a generalization of \cite[Theorem 4.3]{braruf}, is the main result of the present section. For simplicity, we state and prove the result just for open {\it bounded} sets, but it is easily seen that the same proof works for every open set $\Omega\subset\mathbb{R}^N$ such that the embedding $\mathcal{D}^{1,2}_0(\Omega)\hookrightarrow L^q(\Omega)$ is compact. 
\begin{teo}
\label{thm:freeHLE}
Let $1\le q<2$ and let $\Omega\subset\mathbb{R}^N$ be an open bounded set. Then for every $\varphi\in C^\infty_0(\Omega)$ and $\delta>0$ we have 
\begin{equation}
\label{HardyButtazzoq}
\frac{1}{\delta}\,\int_\Omega \left[\left(1-\frac{1}{\delta}\right)\,\left|\frac{\nabla w_{q,\Omega}}{w_{q,\Omega}}\right|^2+\frac{1}{w^{2-q}_{q,\Omega}} \right]\,\varphi^2\,dx\le \int_\Omega |\nabla\varphi|^2\,dx.
\end{equation}
\end{teo}
\begin{proof} 
We recall that
\begin{equation}
\label{conf}
\int_{\Omega}
\left\langle \nabla w_{q,\Omega},\nabla \psi\right\rangle\,dx=\int_{\Omega} w_{q,\Omega}^{q-1}\,\psi\,dx,
\end{equation}
for any $\psi\in \mathcal{D}^{1,2}_0(\Omega)$. 
Let $\varphi\in C^\infty_0(\Omega)$ and let $\varepsilon>0$, by taking in \eqref{conf} the test
function 
\[
\psi=\frac{\varphi^2}{w_{q,\Omega}+\varepsilon},
\]
we get
\begin{equation}
\label{1-7}
\int_\Omega \left[\frac{|\nabla w_{q,\Omega}|^2+w_{q,\Omega}^{q-1}\,(w_{q,\Omega}+\varepsilon)}{(w_{q,\Omega}+\varepsilon)^2} \right]\,\varphi^2\,dx=2\,\int_\Omega \varphi\,\left \langle 
\frac{\nabla w_{q,\Omega}}{(w_{q,\Omega}+\varepsilon)},\nabla \varphi \right\rangle\,dx. 
\end{equation}
By Young's inequality, it holds
\[
\varphi\,\left \langle 
\frac{\nabla w_{q,\Omega}}{(w_{q,\Omega}+\varepsilon)},\nabla \varphi \right\rangle\le \frac{\delta}{2}\,|\nabla \varphi|^2+\frac{1}{2\,\delta}\, \frac{|\nabla w_{q,\Omega}|^2}{(w_{q,\Omega}+\varepsilon)^{2}}\,\varphi^2
\]
for $\delta>0$. Thus we get
\[
\begin{split}
\int_\Omega \left[\frac{|\nabla w_{q,\Omega}|^2+w_{q,\Omega}^{q-1}\,(w_{q,\Omega}+\varepsilon)}{(w_{q,\Omega}+\varepsilon)^2} \right]\,\varphi^2\,dx&\le \delta\,\int_\Omega |\nabla
\varphi|^2\,dx+\frac{1}{\delta}\,\int_\Omega \frac{|\nabla w_{q,\Omega}|^2}{(w_{q,\Omega}+\varepsilon)^{2}}\,\varphi^2\,dx.
\end{split}
\]
The previous inequality gives
\[
\begin{split}
\frac{1}{\delta}\,\int_\Omega \left[\left(1-\frac{1}{\delta}\right)\,\frac{|\nabla w_{q,\Omega}|^2}{(w_{q,\Omega}+\varepsilon)^2} +\frac{w_{q,\Omega}^{q-1}}{(w_{q,\Omega}+\varepsilon)}\right]\,\varphi^2\,dx\le \int_\Omega |\nabla
\varphi|^2\,dx.
\end{split}
\]
By recalling that $\varphi$ is compactly supported in $\Omega$ and observing that\footnote{It is sufficient to remark that $\nabla w_{q,\Omega}\in L^2(\Omega)$ and that by the strong minimum principle, we have 
\[
w_{q,\Omega}\ge c_K>0\,\qquad \mbox{ for every } K\Subset\Omega.
\]}
\[
\left|\frac{\nabla w_{q,\Omega}}{w_{q,\Omega}}\right|^2\in L^1_{\rm loc}(\Omega),
\]
we conclude the proof by taking the limit as $\varepsilon$ goes to $0$ and appealing to the Monotone Convergence Theorem. 
\end{proof}
\begin{oss}[A comment on the proof]
\label{oss:proofhardy}
The idea of the previous proof comes from that of {\it Moser's logarithmic estimate} for elliptic partial differential equations, see \cite[page 586]{moser}. In regularity theory, this is an essential tool in order to establish the validity of Harnack's inequality for solutions. 
\par
An alternative proof is based on Picone's inequality \eqref{pitone}.
This goes as follows: one observes that the function $W=w_{q,\Omega}^{1/\delta}$ locally solves
\[
\begin{split}
-\Delta W&=-\frac{1}{\delta}\,w_{q,\Omega}^{\frac{1}{\delta}-1}\,\Delta w_{q,\Omega}-\frac{1}{\delta}\,\left(\frac{1}{\delta}-1\right)\,w_{q,\Omega}^{\frac{1}{\delta}-2}\,|\nabla w_{q,\Omega}|^2\\
&=W\,\left[\frac{1}{\delta}\,w_{q,\Omega}^{q-2}+\frac{1}{\delta}\,\left(1-\frac{1}{\delta}\right)\,\left|\frac{\nabla w_{q,\Omega}}{w_{q,\Omega}}\right|^2\right].
\end{split}
\]
Thus we have
\[
\int_\Omega \left[\frac{1}{\delta}\,w_{q,\Omega}^{q-2}+\frac{1}{\delta}\,\left(1-\frac{1}{\delta}\right)\,\left|\frac{\nabla w_{q,\Omega}}{w_{q,\Omega}}\right|^2\right]\,W\,\psi\,dx=\int_\Omega \langle \nabla W,\nabla \psi\rangle\,dx,
\]
for every $\psi\in C^\infty_0(\Omega)$. If we now take the test function $\psi=\varphi^2/W$ and use inequality \eqref{pitone}, we get the desired inequality.
\par 
This technique to obtain Hardy-type inequalities is sometimes referred to as {\it Ground State Representation}, see for example \cite[Proposition 1]{Lu}. 
\end{oss}

As a consequence of the Hardy-Lane-Emden inequality, we record the following integrability properties of functions in $\mathcal{D}^{1,2}_0(\Omega)$.
\begin{corollary}
\label{coro:intermedio}
Let $1\le q<2$ and let $\Omega\subset\mathbb{R}^N$ be an open bounded set. Then for every $\varphi\in \mathcal{D}^{1,2}_0(\Omega)$
\begin{equation}
\label{bdss}
\int_\Omega \left|\frac{\nabla w_{q,\Omega}}{w_{q,\Omega}}\right|^2\,\varphi^2\,dx<+\infty\qquad \mbox{ and }\qquad \int_\Omega \frac{\varphi^2}{w_{q,\Omega}^{2-q}}\,dx<+\infty.
\end{equation}
Moreover, if $\{\varphi_n\}_{n\in\mathbb{N}}\subset \mathcal{D}^{1,2}_0(\Omega)$ converges strongly to $\varphi\in\mathcal{D}^{1,2}_0(\Omega)$, then
\[
\lim_{n\to\infty} \int_\Omega \left|\frac{\nabla w_{q,\Omega}}{w_{q,\Omega}}\right|^2\,|\varphi_n-\varphi|^2\,dx=0\qquad \mbox{ and }\qquad \lim_{n\to\infty}\int_\Omega \frac{|\varphi_n-\varphi|^2}{w_{q,\Omega}^{2-q}}\,dx=0.
\]
\end{corollary}
\begin{proof}
Let $\varphi\in\mathcal{D}^{1,2}_0(\Omega)$, then there exists $\{\varphi_n\}_{n\in\mathbb{N}}\subset C^\infty_0(\Omega)$ converging to $\varphi$ in $\mathcal{D}^{1,2}_0(\Omega)$. By choosing $\delta=2$ in \eqref{HardyButtazzoq}, we have
that
\[
\frac{1}{4}\,\int_\Omega \left[\left|\frac{\nabla w_{q,\Omega}}{w_{q,\Omega}}\right|^2+\frac{2}{w^{2-q}_{q,\Omega}} \right]\,\varphi_n^2\,dx\le \int_\Omega |\nabla
\varphi_n|^2\,dx.
\]
By using the norm convergence in the right-hand side and Fatou's Lemma in the left-hand side, we deduce the validity of \eqref{bdss} for $\varphi$.
\par
In order to prove the second part of the statement, we observe that the first part of the proof also implies the validity of inequality \eqref{HardyButtazzoq} in $\mathcal{D}^{1,2}_0(\Omega)$, for $\delta=2$. Plugging in
$\varphi_n-\varphi$ gives that
\[
\limsup_{n\to\infty}\frac{1}{4}\,\int_\Omega \left[\left|\frac{\nabla w_{q,\Omega}}{w_{q,\Omega}}\right|^2+\frac{2}{w^{2-q}_{q,\Omega}} \right]\,|\varphi_n-\varphi|^2\,dx\le \lim_{n\to\infty}\int_\Omega |\nabla \varphi_n-\nabla \varphi|^2\,dx=0,
\] 
as desired.
\end{proof}
As a consequence of Corollary \ref{coro:intermedio} and thanks to the definition of $\mathcal D^{1,2}_0(\Omega)$, we get the following
\begin{corollary}
The Hardy-Lane-Emden inequality \eqref{HardyButtazzoq} is valid for every $\delta>0$ and $u\in\mathcal{D}^{1,2}_0(\Omega)$.
\end{corollary}

\section{Sobolev embeddings and densities}
\label{sec:4}

In this section, we consider general open sets and study the connections between the integrability of $w_{q,\Omega}$ and the embeddings of $\mathcal{D}^{1,2}_0(\Omega)$ into Lebesgue spaces. For the case of the torsion function, i.e. when $q=1$, related studies can be found in \cite{braruf, vaca, bubu} and \cite{bubuve}.
\vskip.2cm
We start with a simple consequence of Theorem \ref{thm:freeHLE}. This is valid for a general open set.
\begin{lm}\label{wHardyButtazzo}
Let $\Omega\subset\mathbb{R}^N$ be an open set and $1\le q<2$. Then for any $\varphi\in C^\infty_0(\Omega)$ it holds
that
\[
\int_{\{x\in\Omega\, :\, w_{q,\Omega}(x)<+\infty\}}\frac{\varphi^2}{w_{q,\Omega}^{2-q}}\,dx\le \int_{\Omega}|\nabla \varphi|^2\,dx.
\]
\end{lm}
\begin{proof}
Let $B_R(0)$ be the ball of radius $R$ centered in $0$, we set $\Omega_R=\Omega\cap B_R(0)$ and $w_R=w_{q,\Omega_R}$. Let $\varphi\in C^\infty_0(\Omega)$, then for every $R$ large enough the support of $\varphi$ is contained in $\Omega_R$. By using \eqref{HardyButtazzoq} on $\Omega_R$ with $\delta=1$, we get
\[
\int_{\Omega}\frac{\varphi^2}{w_R^{2-q}}\,dx\le \int_{\Omega}|\nabla \varphi|^2\,dx.
\]
We conclude by letting $R\to+\infty$ and by Fatou's Lemma. 
\end{proof}

The following result is a generalization of \cite[Theorem 1.2]{braruf}. We point out that the equivalence between $1.$ and $2.$ below is a known fact in Sobolev spaces theory, see \cite[Theorems 15.6.2]{maz}.
\begin{teo}
\label{teo:brarufO}
Let $1\le q<2$ and let $\Omega\subset\mathbb R^N$ be an open set. Then for every $q\le \gamma<2$ the following three facts are equivalent
\begin{itemize}
\item[1.] the embedding $\mathcal D^{1,2}_0(\Omega)\hookrightarrow L^\gamma(\Omega)$ is continuous;
\vskip.2cm
\item[2.] the embedding $\mathcal D^{1,2}_0(\Omega)\hookrightarrow L^\gamma(\Omega)$ is compact;
\vskip.2cm
\item[3.] $w_{q,\Omega}\in L^{\frac{2-q}{2-\gamma}\,\gamma}(\Omega)$.
\end{itemize}
Moreover, we have the double-sided estimates
\begin{equation}
\label{bilat-bdss}
1\le \lambda_{2,\gamma}(\Omega)\,\left(\int_{\Omega} w_{q,\Omega}^{\frac{2-q}{2-\gamma}\,\gamma}\,dx\right)^\frac{2-\gamma}{\gamma}\le \frac{2-\gamma}{\gamma-2\,(q-1)}\,\left(\frac{2-q}{2-\gamma}\right)^2,
\end{equation}
where $\lambda_{2,\gamma}(\Omega)$ is the optimal Poincar\'e constant defined in \eqref{lambda2g}.
\end{teo}
\begin{proof}
As announced above, the equivalence $1.\Longleftrightarrow 2.$ is already known, see also \cite[Theorem 1.2]{braruf} for a different proof. It is sufficient to prove the equivalence $1.\Longleftrightarrow 3.$
\par
Let us suppose that the embedding $\mathcal D^{1,2}_0(\Omega)\hookrightarrow L^\gamma(\Omega)$ is continuous. As always, we set $\Omega_R=\Omega\cap B_R(0)$ and $w_R=w_{q,\Omega_R}$. Then by testing \eqref{wLE} with $w_R^\beta$ for some $\beta \ge 1$, we get  
\[
\begin{split}
\int_{\Omega_R} w_R^{\beta+q-1}\,dx&=\beta\,\left(\frac{2}{\beta+1}\right)^2\, \int_{\Omega_R} \left|\nabla w_R^\frac{\beta+1}{2}\right|^2\,dx\\
&\ge \beta\,\left(\frac{2}{\beta+1}\right)^2\, \lambda_{2,\gamma}(\Omega_R)\,\left(\int_{\Omega_R}w_R^{\frac{\beta+1}{2}\,\gamma}\,dx\right)^\frac{2}{\gamma}\\
&\ge \beta\,\left(\frac{2}{\beta+1}\right)^2\,\lambda_{2,\gamma}(\Omega)\left(\int_{\Omega_R} w_R^{\frac{\beta+1}{2}\,\gamma}\,dx\right)^\frac{2}{\gamma}.
\end{split}
\]
By choosing\footnote{Observe that $\beta\ge 1$ thanks to the fact that $q\le \gamma <2$.} 
\[
\beta=\frac{\gamma-2\,(q-1)}{2-\gamma},
\]
from the previous estimate we get
\[
\left(\int_{\Omega_R} w_R^{\frac{2-q}{2-\gamma}\,\gamma}\,dx\right)^\frac{2-\gamma}{\gamma}\le \frac{2-\gamma}{\gamma-2\,(q-1)}\,\left(\frac{2-q}{2-\gamma}\right)^2\,\frac{1}{\lambda_{2,\gamma}(\Omega)}
\]
By Fatou's Lemma, we can take the limit as $R$ goes to $+\infty$ and get the desired integrability of $w_{q,\Omega}$, together with the upper estimate in \eqref{bilat-bdss}.
\vskip.2cm\noindent
Suppose now that $w_{q,\Omega}\in L^{\frac{2-q}{2-\gamma}\,\gamma}(\Omega)$, this  implies that $w_{q,\Omega}<+\infty$ almost everywhere in $\Omega$. We take $u\in C^\infty_0(\Omega)$, then by H\"older's inequality and Lemma \ref{wHardyButtazzo} we have
\[
\begin{split}
\int_{\Omega}|\varphi|^\gamma\,dx=\int_{\Omega}\frac{|\varphi|^\gamma}{w_{q,\Omega}^{(2-q)\,\frac{\gamma}{2}}}\,w_{q,\Omega}^{(2-q)\,\frac{\gamma}{2}}\,dx&\le \left(\int_{\Omega}\frac{\varphi^2}{w_{q,\Omega}^{2-q}}\,dx\right)^{\frac{\gamma}{2}}\,\left(\int_{\Omega}w_{q,\Omega}^{\frac{2-q}{2-\gamma}\,\gamma}\,dx \right)^\frac{2-\gamma}{2}\\
	&\le \left(\int_{\Omega}|\nabla \varphi|^2\,dx\right)^\frac{\gamma}{2}\,\left(\int_{\Omega} w_{q,\Omega}^{\frac{2-q}{2-\gamma}\,\gamma}\,dx \right)^\frac{2-\gamma}{2}.
\end{split}
\]
We conclude by density of $C^\infty_0(\Omega)$ in $\mathcal D_0^{1,2}(\Omega)$ that the embedding $\mathcal{D}_0^{1,2}(\Omega)\hookrightarrow L^\gamma(\Omega)$ is continuous. Moreover, we also obtain the lower bound in \eqref{bilat-bdss}.
\end{proof}
The following result generalizes \cite[Theorem 1.3]{braruf} and \cite[Theorem 9]{vabu}, by allowing any Lane-Emden densities in place of the torsion function.
\begin{prop}
\label{prop:equivalenze}
Let $1\le q <2$ and let $\Omega\subset\mathbb{R}^N$ be an open set. Then we have that
\[
\lambda_1(\Omega)>0\quad \Longleftrightarrow\quad w_{q,\Omega}\in L^\infty(\Omega).
\]	
Moreover, we have that
\begin{equation}
\label{dorin-estimate}
\lambda_1(\Omega)^\frac{1}{q-2}\le \|w_{q,\Omega}\|_{L^\infty(\Omega)}\le \left(2^N\,\mathcal{C}^2\,\left(\Big(2\,\mathcal{C}\Big)^{2-q}+4\right)\right)^\frac{1}{2-q}\,\lambda_1(\Omega)^\frac{1}{q-2},
\end{equation}
where $\mathcal{C}$ is the same constant appearing in \eqref{brascoMUTO}.
\end{prop}
\begin{proof}
We suppose that $w_{q,\Omega}\in L^\infty(\Omega)$. This in particular implies that $w_{q,\Omega}<+\infty$ almost everywhere in $\Omega$. Then for any $\varphi\in C^\infty_0(\Omega)$ we have that
\[
\begin{split}
\int_{\Omega} \varphi^2\,dx&=\int_{\Omega} \frac{\varphi^2}{w_{q,\Omega}^{2-q}}\,w_{q,\Omega}^{2-q}\,dx\\
&\le \left(\int_{\Omega} \frac{\varphi^2}{w_{q,\Omega}^{2-q}}\,dx\right)\, \|w_{q,\Omega}\|_{L^\infty(\Omega)}^{2-q}\le \|w_{q,\Omega}\|_{L^\infty(\Omega)}^{2-q} \int_{\Omega}|\nabla \varphi|^2\,dx,
\end{split}
\]
the last inequality being due to Lemma \ref{wHardyButtazzo}. This shows that
\[
w_{q,\Omega}\in L^\infty(\Omega)\quad \mbox{ for } 1\le q<2 \qquad \Longrightarrow\qquad \lambda_1(\Omega)>0,
\]	
together with the lower bound in \eqref{dorin-estimate}.
\vskip.2cm\noindent
The converse implication is more involved and we adapt the proof of \cite[Theorem 9]{vabu}, which deals with the case $q=1$. Without loss of generality we can suppose $\Omega$ to be bounded and smooth; indeed, the general case can be then covered by considering a family of smooth bounded sets approaching $\Omega$ from inside.
\par
For ease of notation we set $w:=w_{q,\Omega}$ and we suppose that $w(0)=\|w\|_{L^\infty(\Omega)}$. This can be done up to translating $\Omega$. Moreover we can extend $w$ to $0$ outside $\Omega$. Since $\partial \Omega$ is regular, we get by means of Hopf's Lemma that the extended function, which we still denote by $w$, satisfies
\begin{equation}\label{p0}
-\Delta w\le w^{q-1},
\end{equation}
in the weak sense.
Let $R>0$ to be fixed, and let $\zeta$ be a cut-off Lipschitz function such that 
\[
0\le \zeta\le 1,\qquad \zeta=1 \mbox{ in } B_R(0),\qquad \zeta=0 \mbox{ in } \mathbb{R}^N\setminus B_{2\,R}(0),\qquad  |\nabla\zeta|\le \frac{1}{R}.
\] 
From the variational characterization of $\lambda_1(\Omega)$, we have
\begin{equation}
\label{p1}
\lambda_1(\Omega)\le\frac{\displaystyle\int_{\Omega}|\nabla (w\,\zeta)|^2\,dx}{\displaystyle\int_{\Omega} (w\,\zeta)^2\,dx}=\frac{\displaystyle\int_{\Omega}\Big(|\nabla w|^2\,\zeta^2+2\,w\,\zeta\,\langle\nabla w,\nabla \zeta\rangle+|\nabla\zeta|^2\,w^2\Big)\,dx}{\displaystyle\int_{\Omega} w^2\,\zeta^2\,dx}.
\end{equation}
By using the positive test function $w\,\zeta^2$ into the weak formulation of \eqref{p0}, we get that
\[
\int_{\Omega}|\nabla w|^2\,\zeta^2\,dx+2\,\int_{\Omega}w\,\zeta\,\langle\nabla w,\nabla \zeta\rangle\,dx =\int_{\Omega}\langle\nabla w,\nabla(w\,\zeta^2)\rangle\,dx \le  \int_{\Omega}w^q\,\zeta^2 \,dx.
\]
Thus, by recalling that $w$ attains its maximum in $0$ and using the properties of $\zeta$, from \eqref{p1} we obtain that
\begin{equation}
\label{scaling-enters}
\lambda_1(\Omega)\le \frac{\displaystyle\int_{\Omega}\Big(|\nabla\zeta|^2\,w^2+w^q\,\zeta^2\Big)\,dx}{\displaystyle\int_{\Omega} w^2\,\zeta^2\,dx}\le
2^N\,\omega_N\,\frac{w(0)^2\, R^{N-2}+w(0)^q\,R^N}{\displaystyle\int_{B_R(0)} w^2\,dx}.
\end{equation}
We use now the local $L^\infty-L^2$ estimate of Lemma \ref{lm:milano} to handle the denominator. Indeed, by \eqref{brascoMUTO} with $\alpha=2$ we have that
\[
\int_{B_R(0)} w^2\,dx\ge \omega_N\,R^N\,\left(\frac{1}{\mathcal{C}}\, w(0)-\left(\frac{R}{2}\right)^\frac{2}{2-q}\right)^2.
\]
By spending this information in \eqref{scaling-enters}, we end up with
\[
\lambda_1(\Omega)\le 2^N\,\dfrac{R^{-2}\,w(0)^2+w(0)^q}{\left(\dfrac{1}{\mathcal{C}}\, w(0)-\left(\dfrac{R}{2}\right)^\frac{2}{2-q}\right)^2}.
\]
By choosing 
\[
R=2\,\left(\frac{w(0)}{2\,\mathcal{C}}\right)^{(2-q)/2},
\] 
we  obtain the inequality
\[
\lambda_1(\Omega)\le 2^N\, \frac{4\,\mathcal{C}^2}{w(0)^{2-q}}\,\left(\frac{1}{4}\,\Big(2\,\mathcal{C}\Big)^{2-q}+1\right),
\]
and thus
\[
w(0)\le \left(2^{N}\,\mathcal{C}^2\,\left(\Big(2\,\mathcal{C}\Big)^{2-q}+4\right)\right)^\frac{1}{2-q}\,\lambda_1(\Omega)^\frac{1}{q-2}.
\]
This concludes the proof.
\end{proof}
\begin{oss}[Super-homogeneous embeddings]
\label{oss:superhomo}
A closer inspection of the proof reveals that with exactly the same argument we can prove the following stronger statement: for every $1\le q<2$ and $2\le \gamma<2^*$, we have that
\begin{equation}
\label{LEintegrability-spectrum}
\lambda_{2,\gamma}(\Omega)>0\quad \Longleftrightarrow\quad w_{q,\Omega}\in L^\infty(\Omega),
\end{equation}	
where 
\[
2^*=\frac{2\,N}{N-2}, \ \mbox{ for } N\ge 3\qquad \mbox{ and }\qquad 2^*=+\infty,\ \mbox{ for } N\in\{1,\,2\}.
\]
Observe that \eqref{LEintegrability-spectrum} implies in particular that
\begin{equation}
\label{super-homo-embds}
\mathcal{D}^{1,2}_0(\Omega) \hookrightarrow L^2(\Omega)\qquad \Longleftrightarrow \qquad \mathcal{D}^{1,2}_0(\Omega) \hookrightarrow L^\gamma(\Omega), \quad \mbox{ for } 2<\gamma<2^*.
\end{equation}
For the implication $\Longrightarrow$, it is sufficient to reproduce the proof above, using the variational characterization of $\lambda_{2,\gamma}(\Omega)$ and the $L^\infty$ estimate \eqref{brascoMUTO}, this time with $\alpha=\gamma$. 
\par
For the converse implication, it is sufficient to use the {\it Gagliardo-Nirenberg interpolation inequality} (see for example \cite[Proposition 2.6]{braruf})
\[
\left(\int_{\mathbb{R}^N} |u|^\gamma\,dx\right)^\frac{1}{\gamma}\le C\,\left(\int_{\mathbb{R}^N} |u|^2\,dx\right)^\frac{1-\vartheta}{2}\,\left(\int_{\mathbb{R}^N}|\nabla u|^2\,dx\right)^\frac{\vartheta}{2},
\]
where $C=C(N,\gamma)>0$ and
\[
\vartheta=\left(1-\frac{2}{\gamma}\right)\,\frac{N}{2},\qquad 2<\gamma<2^*.
\]
This shows that if $\mathcal{D}^{1,2}_0(\Omega) \hookrightarrow L^2(\Omega)$ is continuous, then $\mathcal{D}^{1,2}_0(\Omega) \hookrightarrow L^\gamma(\Omega)$ is continuous as well. We leave the details to the interested reader.
\par
We point out that the equivalence \eqref{super-homo-embds} can also be found in \cite[Theorem 15.4.1]{maz}. The proof there is different. 
\end{oss}
We conclude this section with the following simple result which we record for completeness.
\begin{prop}
Let $1\le q<2$ and let $\Omega\subset\mathbb{R}^N$ be an open set such that the embedding $\mathcal{D}^{1,2}_0(\Omega)\hookrightarrow L^q(\Omega)$ is continuous. Then the embedding $\mathcal{D}^{1,2}_0(\Omega)\hookrightarrow L^2(\Omega)$ is compact.
\end{prop}
\begin{proof}
We already know by Theorem \ref{teo:brarufO} that the continuity of the embedding $\mathcal{D}^{1,2}_0(\Omega)\hookrightarrow L^q(\Omega)$ is equivalent to its compactness. Then it is sufficient to use the Gagliardo-Nirenberg inequality 
\[
\left(\int_{\mathbb{R}^N} |u|^2\,dx\right)^\frac{1}{2}\le C\,\left(\int_{\mathbb{R}^N} |u|^q\,dx\right)^\frac{1-\vartheta}{q}\,\left(\int_{\mathbb{R}^N}|\nabla u|^2\,dx\right)^\frac{\vartheta}{2},
\]
where $C=C(N,q)>0$ and
\[
\vartheta=\left(1-\frac{q}{2}\right)\,\frac{2\,N}{(2-q)\,N+2\,q}.
\]
This guarantess that every bounded sequence $\{u_n\}_{n\in\mathbb{N}}\subset\mathcal{D}^{1,2}_0(\Omega)$ strongly converging in $L^q(\Omega)$, strongly converges in $L^2(\Omega)$ as well.
This gives the desired conclusion.
\end{proof}
\begin{oss}
The converse implication of the previous proposition does not hold. Indeed, let $\{r_i\}_{i\in\mathbb{N}}\subset \mathbb{R}$ be a sequence of strictly positive numbers, such that
\[
\lim_{i\to\infty} r_i=0 \qquad \mbox{ and }\qquad \sum_{i=0}^\infty r_i^{\frac{2}{2-\gamma}+N}=+\infty, \mbox{ for every } 1\le \gamma<2.
\] 
For example, one could take $r_i=1/\log(2+i)$.
We then define the sequence of points $\{x_i\}_{i\in\mathbb{N}}\subset\mathbb{R}^N$ by
\[
\left\{\begin{array}{rcl}
x_0&=&(0,\dots,0),\\
x_{i+1}&=&(r_i+r_{i+1},0,\dots,0)+x_i,
\end{array}
\right.
\]
and the disjoint union of balls
\[
\Omega=\bigcup_{i=0}^\infty B_{r_i}(x_i).
\]
Thanks to the choice of the radii $r_i$ we have
\[
w_\Omega\in L^\infty(\Omega)\qquad \mbox{ and }\qquad \begin{array}{c}
\mbox{for every $\varepsilon>0$, there exists $R>0$}\\
\mbox{ such that }\ \|w_\Omega\|_{L^\infty(\Omega\setminus B_R)}<\varepsilon,
\end{array}
\]
thus the embedding $\mathcal{D}^{1,2}_0(\Omega)\hookrightarrow L^2(\Omega)$ is compact, see \cite[Theorem 1.3]{braruf}.
\par
On the other hand, $w_\Omega\not\in L^\gamma(\Omega)$ for every $\gamma\in[1,+\infty)$ (see \cite[Example 5.2]{braruf}). Thus, by Theorem~\ref{teo:brarufO}, $\mathcal{D}^{1,2}_0(\Omega)$ is not continuously embedded in any $L^\gamma(\Omega)$, with $1\le \gamma<2$.
\end{oss}

\section{Hardy-Lane-Emden inequalities for sets with positive spectrum}
\label{sec:5}
We want to generalize Theorem \ref{thm:freeHLE} and prove that the Hardy-Lane-Emden inequality \eqref{HardyButtazzoq} holds true on any open set $\Omega\subset\mathbb{R}^N$ with positive spectrum, i.e. such that the constant $\lambda_1(\Omega)$ defined in \eqref{lambda1} is positive.
\par
We need an expedient result which has some interest by itself.
\begin{prop}
\label{lm:berry}
Let $1\le q<2$ and let $\Omega\subset\mathbb{R}^N$ be an open set with positive spectrum. Then $\nabla w_{q,\Omega}\in L^2_{{\rm loc}}(\Omega)$ and $w_{q,\Omega}$ is a local weak solution of the Lane-Emden equation
\begin{equation}
\label{lm:berry-eqn}	
-\Delta w_{q,\Omega}=w_{q,\Omega}^{q-1},
\end{equation}
i.e. we have
\[
\int_\Omega \langle \nabla w_{q,\Omega},\nabla \varphi\rangle\,dx=\int_\Omega w^{q-1}_{q,\Omega}\,\varphi\,dx,\qquad \mbox{ for every }\varphi\in \mathcal{D}^{1,2}_0(\Omega') \mbox{ and }\Omega'\Subset\Omega. 
\]
\end{prop}
\begin{proof} 
Let $w_R$ be the Lane-Emden function of $\Omega\cap B_R(0)$ and let $\Omega'\Subset\Omega$. We aim to show that there exists a constant $C>0$ such that
\begin{equation}
\label{bds-distrib} 
\int_{\Omega'}|\nabla w_R|^2\,dx\le C,\qquad \mbox{ for every } R>0.
\end{equation}
Indeed, this entails that $\nabla w_R$ weakly converge (up to extracting a sequence) in $L^2(\Omega')$ to a vector field $Z\in L^2(\Omega')$. The assumption $\lambda_1(\Omega)>0$ 
implies that $w_{q,\Omega}\in L^\infty(\Omega)$, by Proposition \ref{prop:equivalenze}. Then by recalling that $0\le w_R\le w_{q,\Omega}$, it is not difficuly to see that $Z$ must coincide with the distributional gradient $\nabla w_{q,\Omega}$ (see for example \cite[Proposition 3.6]{braruf}).
\par
In particular, for every $\varphi\in C^\infty_0(\Omega')$ the identity
\[
\int_{\Omega}\langle \nabla w_R,\nabla\varphi\rangle\,dx=\int_{\Omega}w_R^{q-1}\varphi\,dx 
\]
passes to the limit and we are done. The fact that we can allow test functions $\varphi\in \mathcal{D}^{1,2}_0(\Omega')$ follows by density. 
\par
Thus we are left to show that \eqref{bds-distrib} holds true. Let $\Omega'\Subset\Omega''\Subset\Omega$ and take $\eta\in C^\infty_0(\Omega'')$ a standard cut-off function, with $\eta=1$ on $\Omega'$ and $|\nabla\eta|\le C/\,{\rm dist}(\Omega',\Omega'')$. Then, for $R>0$ large enough, we test the Lane-Emden equation satisfied by $w_R$ with $\varphi=w_R\,\eta^2\,$. This yields
\[
\begin{aligned}
\int_{\Omega} |\nabla w_R|^2\,\eta^2\,dx
&=\int_\Omega\eta^2\,w_R^{q}\,dx-2\int_{\Omega}w_R\,\eta\,\langle \nabla w_R,\nabla \eta\rangle\,dx\\
&\le \int_\Omega\eta^2\,w_R^{q}\,dx+\frac12\int_{\Omega} |\nabla w_R|^2\,\eta^2\,dx+2\,\int_{\Omega} |\nabla \eta|^2w_R^2\,dx.
\end{aligned}
\]
Since by construction we have $w_R\le w_{q,\Omega}$, we deduce that
\[
\int_{\Omega'} |\nabla w_R|^2\,dx\le 2\,\int_{\Omega''}w_{q,\Omega}^q\,dx+\frac{4\,C}{{\rm dist}(\partial\Omega',\partial\Omega'')}\int_{\Omega''}w_{q,\Omega}^2\,dx.
\]
By recalling that $w_{q,\Omega}\in L^\infty(\Omega)$,
we get \eqref{bds-distrib} from the previous estimate.
\end{proof}
\begin{teo}
\label{teo:genset}
Let $1\le q<2$ and let $\Omega\subset\mathbb{R}^N$ be an open set with positive spectrum. Then for every $u\in C^\infty_0(\Omega)$ and $\delta>0$ we still have \eqref{HardyButtazzoq}.
\end{teo}
\begin{proof}
Let $\Omega'\Subset\Omega$ and $\varphi\in C^\infty_0(\Omega')$. Since $w_{q,\Omega}\in H^1(\Omega')$ by the previous result, we can use $\varphi^2/(w_{q,\Omega}+\varepsilon)$ as a test function in \eqref{lm:berry-eqn}. Then we can repeat word by word the proof of Theorem \ref{thm:freeHLE} to show that \eqref{HardyButtazzoq} holds for any $\varphi\in C^\infty_0(\Omega')$. The conclusion then follows by arbitrariness of $\Omega'\Subset\Omega$.
\end{proof}

\section{Lower bounds for the ground state energy}
\label{sec:6}
For a negative potential $V\in L^2_{\rm loc}(\Omega)$, we go back to our initial task and consider the operator $\mathcal{H}_{\Omega,V}=-\Delta+V$. We already observed that $\mathcal{H}_{\Omega,V}$ is symmetric and self-adjoint, with domain $\mathfrak{D}(\mathcal{H}_{\Omega,V})$ defined in \eqref{dominio}. 
We recall the notation from the Introduction
\[
\mathcal{Q}_{\Omega,V}(\varphi)=\int_{\Omega} |\nabla \varphi|^2\,dx+\int_\Omega V\,\varphi^2\,dx,\qquad \varphi\in \mathfrak{D}(\mathcal{H}_{\Omega,V}),
\]
and we set
\[
\lambda_1(\Omega;V)=\inf_{u\in C^\infty_0(\Omega)} \left\{\mathcal{Q}_{\Omega,V}(\varphi)\, :\, \int_{\Omega} \varphi^2\,dx=1\right\}. 
\]
We need the following expedient result which asserts that under suitable assumptions on the potential $V$, the infimum in the definition of $\lambda_1(\Omega;V)$ can be equivalently
taken upon $\mathcal{D}^{1,2}_0(\Omega)$.
\begin{lm}
\label{inf_BL}
Let $\Omega\subset\mathbb{R}^N$ be an open set with positive spectrum,
and let $V\in L^2_{\rm loc}(\Omega)$ be a negative potential. 
We further suppose that there exists a constant $C>0$ such that
\begin{equation}
\label{Vdual}
\int_\Omega |V|\,\varphi^2\,dx\le C\int_\Omega |\nabla \varphi|^2\,dx,\qquad \mbox{ for every }\varphi\in C^\infty_0(\Omega).
\end{equation}
Then
\[
\lambda_1(\Omega;V)=\inf_{\varphi\in \mathcal{D}^{1,2}_0(\Omega)} \left\{\mathcal{Q}_{\Omega,V}(\varphi)\, :\, \int_{\Omega} \varphi^2\,dx=1\right\}.
\]
\end{lm}
\begin{proof}
Since $C^\infty_0(\Omega)\subset\mathcal{D}^{1,2}_0(\Omega)$, it is straightforward to see that
\[
\lambda_1(\Omega;V)\ge \inf_{\varphi\in \mathcal{D}^{1,2}_0(\Omega)} \left\{\mathcal{Q}_{\Omega,V}(\varphi)\, :\, \int_{\Omega} \varphi^2\,dx=1\right\}.
\]
In order to prove the reverse inequality, we take $\varphi \in\mathcal{D}^{1,2}_0(\Omega)$ with unit $L^2$ norm and a sequence $\{\varphi_n\}_{n\in\mathbb{N}}\subset C^\infty_0(\Omega)$ converging to $\varphi$ in $\mathcal{D}^{1,2}_0(\Omega)$. Observe that since $\lambda_1(\Omega)>0$, this in particular implies that $\{\varphi_n\}_{n\in\mathbb{N}}$ converges strongly in $L^2(\Omega)$ as well, by Poincar\'e inequality. From the definition of $\lambda_1(\Omega;V)$, we obtain that
\[
\lambda_1(\Omega;V)\le \frac{\displaystyle\int_\Omega |\nabla
\varphi_n|^2\,dx+\int_\Omega V\,\varphi_n^2\,dx}{\displaystyle \int_\Omega \varphi_n^2\,dx},\qquad \mbox{ for every }n\in\mathbb{N}.
\]
We observe that 
\[
\lim_{n\to\infty}\int_\Omega |\nabla \varphi_n|^2\,dx=\int_\Omega |\nabla
\varphi|^2\,dx\qquad \mbox{ and }\qquad \lim_{n\to\infty}\int_\Omega 
\varphi_n^2\,dx=1.
\]
In order to handle the term containing $V$, we first observe that by Fatou's Lemma and density of $C^\infty_0(\Omega)$ in $\mathcal{D}^{1,2}_0(\Omega)$, inequality \eqref{Vdual} extends to the whole $\mathcal{D}^{1,2}_0(\Omega)$.
Then we use that
\[
\begin{split}
\left|\int_\Omega V\,\varphi_n^2\,dx-\int_\Omega V\,\varphi^2\,dx\right|&\le \left(\int_{\Omega} |V|\,(\varphi_n-\varphi)^2\,dx\right)^\frac{1}{2}\,\left(\int_{\Omega} |V|\,(\varphi_n+\varphi)^2\,dx\right)^\frac{1}{2}\\
&\le C\,\left(\int_{\Omega} |\nabla \varphi_n-\nabla \varphi|^2\,dx\right)^\frac{1}{2}\\
&\times\left[\left(\int_{\Omega} |\nabla \varphi_n|^2\,dx\right)^\frac{1}{2}+\left(\int_{\Omega} |\nabla \varphi|^2\,dx\right)^\frac{1}{2}\right],
\end{split}
\]
thanks to H\"older and Minkowski inequalities, together with the hypothesis on $V$. If we use the convergence in $\mathcal{D}^{1,2}_0(\Omega)$, we obtain that
\[
\lim_{n\to\infty} \int_\Omega V\,\varphi_n^2\,dx=\int_\Omega V\,\varphi^2\,dx,
\]
which gives the desired conclusion.
\end{proof}
The following is the main result of the paper.
\begin{teo}
\label{teo:mainthm}
Let $\Omega\subset\mathbb{R}^N$ be an open set with positive spectrum
and let $V\in L^2_{\rm loc}(\Omega)$. For  an exponent $1\le q<2$, we suppose that
\begin{equation}
\label{KJV}
0\ge V\ge -\frac{1}{4}\,\left|\frac{\nabla w_{q,\Omega}}{w_{q,\Omega}}\right|^2,\qquad \mbox{ a.\,e. in }\Omega.
\end{equation}
Then the spectrum $\sigma(\mathcal{H}_{\Omega,V})$ of $\mathcal{H}_{\Omega,V}$ is positive and we have
\[
\inf\sigma(\mathcal{H}_{\Omega,V})=\lambda_1(\Omega;V)\ge \frac{1}{2}\,\|w_{q,\Omega}\|^{q-2}_{L^\infty(\Omega)}.
\]
\end{teo}
\begin{proof}
We prove separately that 
\[
\lambda_1(\Omega;V)\ge \frac{1}{2}\,\|w_{q,\Omega}\|^{q-2}_{L^\infty(\Omega)}\qquad \mbox{ and }\qquad \inf\sigma(\mathcal{H}_{\Omega,V})=\lambda_1(\Omega;V).
\]
We first observe that assuming $\lambda_1(\Omega)>0$, implies the validity of the Hardy-Lane-Emden inequality
\begin{equation}
\label{teo:mainthm:LEineq}
\frac{1}{4}\,\int_\Omega \left|\frac{\nabla w_{q,\Omega}}{w_{q,\Omega}}\right|^2\,\varphi^2\,dx+\frac{1}{2}\,\int_\Omega \frac{\varphi^2}{w_{q,\Omega}^{2-q}}\,dx\le \int_\Omega |\nabla
\varphi|^2\,dx,\qquad \mbox{ for }\varphi\in C^\infty_0(\Omega).
\end{equation}
Indeed, this follows from Theorem \ref{teo:genset} with $\delta=2$. Thanks to hypothesis \eqref{KJV}, we thus obtain 
\[
\frac{1}{2}\,\int_\Omega \frac{\varphi^2}{w_{q,\Omega}^{2-q}}\,dx\le \int_\Omega |\nabla
\varphi|^2\,dx+\int_\Omega V\,\varphi^2\,dx,\qquad \mbox{ for } \varphi\in C^\infty_0(\Omega).
\]
Also observe that $w_{q,\Omega}\in L^\infty(\Omega)$, thanks to Proposition \ref{prop:equivalenze}. 
In particular, we get the following lower bound for the quadratic form
\[
\mathcal{Q}_{\Omega,V}(\varphi)\ge \frac{1}{2\,\|w_{q,\Omega}\|^{2-q}_{L^\infty(\Omega)}},\qquad \mbox{ for every }\varphi\in C^\infty_0(\Omega) \mbox{ with } \int_\Omega \varphi^2\,dx=1.
\] 
By arbitrariness of $\varphi$, this gives the lower bound on $\lambda_1(\Omega;V)$.
\vskip.2cm\noindent
We now prove that
\[
\inf\sigma(\mathcal{H}_{\Omega,V})=\lambda_1(\Omega;V).
\]
To see this, we first observe that by self-adjointness (see \cite[Theorem 2.20]{Te}) we have that
\[
\inf \sigma(\mathcal{H}_{\Omega,V})=\inf_{\varphi\in \mathfrak{D}(\mathcal{H}_{\Omega,V})}\left\{\mathcal{Q}_{\Omega,V}(\varphi)\, :\, \int_\Omega \varphi^2\,dx=1\right\}.
\]
By recalling that $C^\infty_0(\Omega)\subset \mathfrak{D}(\mathcal{H}_{\Omega,V})$, this immediately gives
\[
\inf \sigma(\mathcal{H}_{\Omega,V})\le \inf_{\varphi\in C^\infty_0(\Omega)}\left\{\mathcal{Q}_{\Omega,V}(\varphi)\, :\, \int_\Omega \varphi^2\,dx=1\right\}=\lambda_1(\Omega;V).
\] 
In order to prove the reverse inequality, we make use of Lemma \ref{inf_BL}. For this, we need to prove that our potentials $V$ satisfy \eqref{Vdual}. But this easily follows from \eqref{teo:mainthm:LEineq} and \eqref{KJV}, which gives that
\begin{equation}
\label{1!}
\int_\Omega |V|\,\varphi^2\,dx\le \int_\Omega |\nabla \varphi|^2\,dx,\qquad \mbox{ for every } \varphi\in C^\infty_0(\Omega).
\end{equation}
We can thus apply Lemma \ref{Vdual} and obtain that
\[
\begin{split}
\lambda_1(\Omega;V)&=\inf_{\varphi\in \mathcal{D}^{1,2}_0(\Omega)}\left\{\mathcal{Q}_{\Omega,V}(\varphi)\, :\, \int_\Omega \varphi^2\,dx=1\right\}\\
&\le \inf_{\varphi\in \mathfrak{D}(\mathcal{H}_{\Omega,V})}\left\{\mathcal{Q}_{\Omega,V}(\varphi)\, :\, \int_\Omega \varphi^2\,dx=1\right\}=\inf \sigma(\mathcal{H}_{\Omega,V}),
\end{split}
\]
where we also used that $\mathfrak{D}(\mathcal{H}_{\Omega,V})\subset H^1_0(\Omega)=\mathcal{D}^{1,2}_0(\Omega)$, see Remark \ref{oss:uguali}. This concludes the proof.
\end{proof}
By using the $L^\infty$ estimate of Proposition \ref{prop:equivalenze}, we also get the following explicit lower bound on $\lambda_1(\Omega;V)$, in terms of a dimensional constant and $\lambda_1(\Omega)$.
\begin{corollary}
\label{coro:teo:mainthm}
Under the assumptions of Theorem \ref{teo:mainthm}, we also have
\begin{equation}\label{ans-ref}
\inf\sigma(\mathcal{H}_{\Omega,V})=\lambda_1(\Omega;V)\ge \frac{1}{2}\,\frac{1}{2^N\mathcal{C}^2\,\left(\Big(2\,\mathcal{C}\Big)^{2-q}+4\right)}\,\lambda_1(\Omega),
\end{equation}
where $\mathcal{C}>0$ is the same constant appearing in \eqref{brascoMUTO}.
\end{corollary}

\begin{oss}[On the sharpness of the bound]
\label{oss:referee}	
Let $\Omega\subset\mathbb{R}^N$ be open with positive spectrum and let $h\in L^\infty(\Omega)$ be a {\it nonnegative} function. It is easy to see that if $V$ is as in Theorem \ref{teo:mainthm}, then the perturbed potential $V-h$ still verifies hypothesis \eqref{Vdual} of Lemma \ref{inf_BL}. Indeed, for every $\varphi\in C^\infty_0(\Omega)$ we have
\[
\begin{split}
\int_\Omega |V-h|\,\varphi^2\,dx&=\int_\Omega (-V)\,\varphi^2\,dx+\int_\Omega h\,\varphi^2\,dx\\
&\le \int_\Omega |\nabla \varphi|^2\,dx+\|h\|_{L^\infty(\Omega)}\,\int_\Omega \varphi^2\,dx\le \left(1+\frac{\|h\|_{L^\infty(\Omega)}}{\lambda_1(\Omega)}\right)\int_\Omega |\nabla \varphi|^2\,dx,
\end{split}
\]
where we used \eqref{1!} and the fact that $\Omega$ has positive spectrum.
Thus we get	
\[
\begin{split}
\lambda_1(\Omega;V-h)&=\inf_{\varphi\in \mathcal{D}^{1,2}_0(\Omega)} \left\{\mathcal{Q}_{\Omega,V-h}(\varphi)\, :\, \int_{\Omega} \varphi^2\,dx=1\right\} \\
&\ge \inf_{\varphi\in \mathcal{D}^{1,2}_0(\Omega)} \left\{\mathcal{Q}_{\Omega,V}(\varphi)\, :\, \int_{\Omega} \varphi^2\,dx=1\right\} -\|h\|_{L^\infty(\Omega)}\\
&\ge \lambda_1(\Omega;V)-\|h\|_{L^\infty(\Omega)}.
\end{split}
\]
In view of the quantitative bound \eqref{ans-ref}, the spectrum $\sigma(\mathcal H_{\Omega,V-h})$ remains positive for example if the bounded perturbation $-h$ is such that
\[
\|h\|_{L^\infty(\Omega)}<\frac{1}{2}\,\frac{1}{2^N\mathcal{C}^2\,\left(\Big(2\,\mathcal{C}\Big)^{2-q}+4\right)}\,\lambda_1(\Omega).
\]	
%
In other words, we have room to translate downward the potential $V$ and still guarantee that the spectrum stays positive.
\end{oss}

\begin{oss}[The choice of $\delta$]
The result of Theorem \ref{teo:mainthm} follows by chosing $\delta=2$ in \eqref{HardyButtazzoq}. One may wonder why we limited ourselves to this choice only.
In order to clarify this point, we start by rewriting \eqref{HardyButtazzoq} as
\[
\frac{1}{\delta}\,\int_\Omega \frac{|u|^2}{w_{q,\Omega}^{2-q}}\,dx\le \int_\Omega |\nabla
u|^2\,dx+\frac{1}{\delta}\left(\frac{1}{\delta}-1\right)\,\int_\Omega \left|\frac{\nabla w_{q,\Omega}}{w_{q,\Omega}}\right|^2\,|u|^2\,dx.
\]
This implies that for every potential $V$ such that
\begin{equation}
\label{lowerbound}
V\ge \frac{1}{\delta}\left(\frac{1}{\delta}-1\right)\,\left|\frac{\nabla w_{q,\Omega}}{w_{q,\Omega}}\right|^2,\qquad \mbox{ a.\,e. in }\Omega,
\end{equation}
we have that
\[
\frac{1}{\delta}\,\int_\Omega \frac{|u|^2}{w_{q,\Omega}^{2-q}}\,dx\le \int_\Omega |\nabla
u|^2\,dx+\int_\Omega V\,|u|^2\,dx.
\]
In particular, we get the following lower bound
\[
\mathcal{Q}_{\Omega,V}(\varphi)\ge \frac{1}{\delta\,\|w_{q,\Omega}\|^{2-q}_{L^\infty(\Omega)}},\qquad \mbox{ for every }\varphi\in C^\infty_0(\Omega) \mbox{ with } \int_\Omega \varphi^2\,dx=1.
\] 
Observe that the right-hand side in \eqref{lowerbound} is pointwise minimal when $\delta=2$. This explains our choice.
\end{oss}

\section{Applications}
\label{sec:7}

In this section, we compute the limit potential appearing in \eqref{lowerbound} in some particular cases and give the relevant lower bound on the ground state energy $\lambda_1(\Omega;V)$.
In the following examples we take $q=1$, i.e. we use the torsion function.
\begin{figure}

\begin{tikzpicture}[>=latex,xscale=2,yscale=2]
\draw[->] (-2.3,0) -- (2.5,0) node[below] {$|x|$};
\foreach \x in {-2,...,-1,1,2,...,2}
\draw[shift={(\x,0)}] (0pt,2pt) -- (0pt,-2pt) node[below] {\tiny $\x$};
\foreach \x in {-23,...,-1,1,2,...,23}
\draw[shift={(\x/10,0)}] (0pt,1pt) -- (0pt,-1pt);
\foreach \x in {-2,-1.5,-0.5,0.5,1.5,2}
\draw[shift={(\x,0)}] (0pt,2pt) -- (0pt,-2pt) node[below] {\tiny $\x$};

\draw[->] (0,-2.5) -- (0,.4) node[left] {$y$};
\foreach \y in {-2.5,-2,-1.5,-1,-0.5}
\draw[shift={(0,\y)}] (2pt,0pt) -- (-2pt,0pt) node[left] {\tiny $\y$};
\foreach \y in {-24,...,-1,1,2}
\draw[shift={(0,\y/10)}] (-1pt,0pt) -- (1pt,0pt);

\draw[blue, thick, domain=-.7327:.7327] plot (\x, {-pow(abs(\x),2)/(pow(1-abs(\x)*abs(\x),2))});
\draw[dashed] (-1,0)--(-1,-2.5);
\draw[dashed] (1,0)--(1,-2.5);

\node[right,blue] at (1.05,-1.25) {$V(x)=-\dfrac{|x|^2}{(1-|x|^2)^2_+}$};
\end{tikzpicture}

\caption{The limit potential in a ball of radius $1$.}

\end{figure}
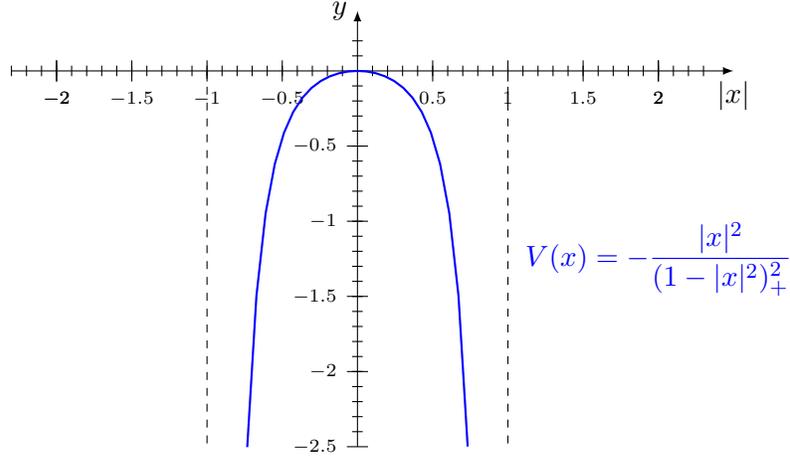

\subsection{$N-$dimensional ball}

Let us take $\Omega=B_1(0)\subset\mathbb{R}^N$, then
\[
w_\Omega(x)=\frac{1-|x|^2}{2\,N},\qquad x\in B_1(0),
\]
and thus
\[
-\frac{1}{4}\,\left|\frac{\nabla w_\Omega}{w_\Omega}\right|^2=-\frac{|x|^2}{(1-|x|^2)^2}.
\]
Thus for every $V\in L^2_{\rm loc}(\Omega)$ such that
\[
0\ge V\ge -\frac{|x|^2}{(1-|x|^2)^2},
\]
from Theorem \ref{teo:mainthm} we get
\[
\lambda_1(\Omega;V)\ge \frac{1}{2\,\|w_\Omega\|_{L^\infty(\Omega)}}=N.
\]

\subsection{An infinite slab} 
We now consider the set $\Omega=(-1,1)\times\mathbb{R}^{N-1}$. We first need to compute its torsion function. This is the content of the next
\begin{lm}
Let $\Omega=(-1,1)\times \mathbb{R}^{N-1}\subset\mathbb{R}^N$. Then its torsion function is given by
\[
w_\Omega(x_1,x')=\frac{1-x_1^2}{2},\qquad (x_1,x')\in(-1,1)\times\mathbb{R}^{N-1}.
\]
\end{lm}
\begin{proof}
We set
\[
Q_R=(-1,1)\times (-R,R)^{N-1},
\]
then we notice that 
\[
w_\Omega=\lim_{R\to+\infty} w_{Q_R}.
\]
That is, we can approximate $\Omega$ by the sets $Q_R$ and not only by $\Omega\cap B_R(0)$, in order to construct $w_\Omega$. This follows since 
\[
w_\Omega=\lim_{R\to+\infty}w_{\Omega\cap B_R(0)}=\lim_{R\to+\infty}w_{\Omega\cap B_{\sqrt{N}R}(0)},
\]
and the fact that by the comparison principle 
\[
w_{\Omega\cap B_R(0)}\le w_{Q_R}\le w_{\Omega\cap B_{\sqrt{N}R}(0)},
\]
for $R\gg 1$.
Let  
\[
w(x,x')=\frac{1-x_1^2}{2},\qquad (x_1,x')\in(-1,1)\times\mathbb{R}^{N-1},
\] 
and notice that $w$ is a classical solution in $\Omega$ of $-\Delta w=1$, vanishing on $\partial\Omega$. 
\par
Observe that $w\ge w_{Q_R}$ for any $R>0$, thanks to the comparison principle. Thus 
\[
w\ge w_\Omega=\lim_{R\to+\infty} w_{Q_R}.
\] 
To get the reverse inequality, we observe that, again by the comparison principle, $w_{Q_R}\ge w_{\mathcal E_R}$. Here $\mathcal E_R$ is the ellipsoid inscribed in $Q_R$, given by
\[
\mathcal E_R=\left \{(x_1,x')\in \mathbb{R}^N\, :\, x_1^2+\frac{|x'|^2}{R^2}=1 \right\},
\]
and it is immediate to check that 
\[
w_{\mathcal E_R}=\frac{R^2}{R^2+(N-1)}\,\frac{1-x_1^2-\dfrac{|x'|^2}{R^2}}{2}.
\]
This gives
\[
w_\Omega\ge \lim_{R\to+\infty}w_{\mathcal E_R}=w,
\] 
and thus the desired conclusion.
\end{proof}
Let us take $\Omega=(-1,1)\times \mathbb{R}^{N-1}$, then
\[
\frac{1}{4}\,\left|\frac{\nabla w_\Omega}{w_\Omega}\right|^2=\frac{x_1^2}{(1-x_1^2)^2}.
\]
Thus for every potential $V\in L^2_{\rm loc}(\Omega)$ such that
\[
0\ge V(x_1,x')\ge -\frac{x_1^2}{(1-x_1^2)^2},
\]
still by Theorem \ref{teo:mainthm} we get
\[
\lambda_1(\Omega;V)\ge \frac{1}{2\,\|w_\Omega\|_{L^\infty(\Omega)}}=1.
\]

\begin{figure}[!h]
\includegraphics[scale=.8]{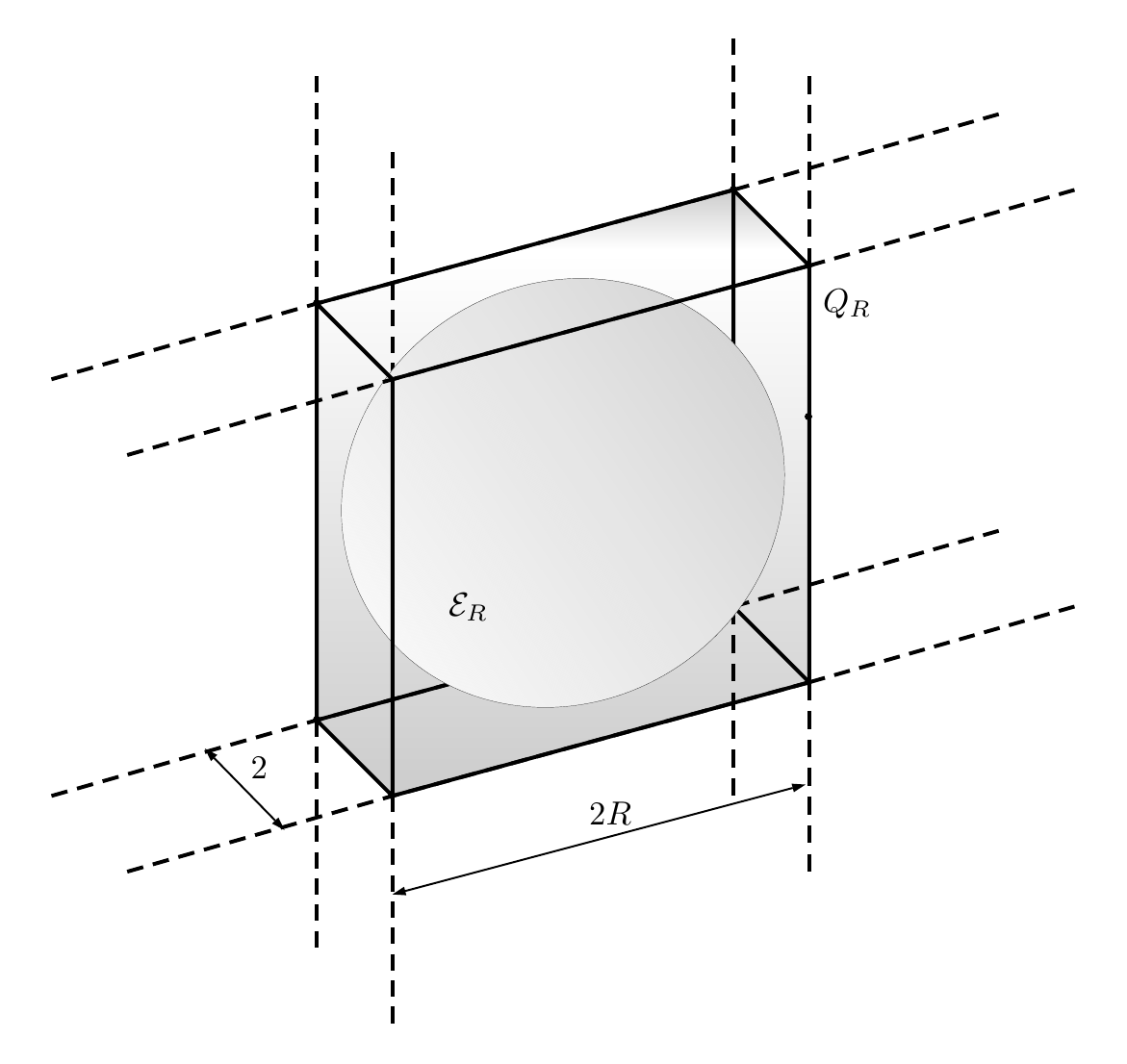}
\caption{Approximating an infinite slab.}
\end{figure}

\subsection{A rectilinear wave-guide}

Finally, we want to consider a set of the form $\Omega=\omega\times\mathbb{R}$, where $\omega\subset\mathbb{R}^{N-1}$ is an open bounded set with Lipschitz boundary. Again, we first identify its torsion function.
\begin{lm}
\label{exa:rectiWG}
Let $\Omega=\omega\times \mathbb{R}\subset\mathbb{R}^N$. Then its torsion function is given by
\begin{equation}
\label{dimreduct}
w_\Omega(x',x_N)=w_\omega(x'),\qquad (x',x_N)\in\omega\times\mathbb{R}
\end{equation}
where $w_\omega$ stands for the torsion function of the set $\omega$ in $\mathbb R^{N-1}$.
\end{lm}
\begin{proof}
We divide the proof in four steps.
\vskip.2cm\noindent
{\bf Step 1.} In this step, we prove that for every $\ell>0$
\[
w_\Omega(x',x_N+\ell)=w_\Omega(x',x_N),\qquad (x',x_N)\in\omega\times\mathbb{R},
\] 
i.e. the torsion function does not depend on the $x_N$ variable.
\par 
To see this, let us suppose for simplicity that 
\(
\omega\subset (-R_0,R_0)^{N-1},
\)
and take $R\ge R_0$. We set $Q_R=\Omega\cap (-R,R)^N$ and $w_R=w_{Q_R}$. Then
\begin{equation}
\label{uan}
w_\Omega(x',x_N+\ell)=\lim_{R\to+\infty} w_{R}(x',x_N+\ell).
\end{equation}
We now observe that if we further set 
\(
Q_{R,\ell}=\Omega\cap \big((-R,R)^{N-1}\times(-R-\ell,R-\ell)\big)
\) 
and $w_{R,\ell}=w_{Q_{R,\ell}}$ then by construction we have that
\begin{equation}
\label{ciu}
w_{R,\ell}(x',x_N)=w_R(x',x_N+\ell).
\end{equation}
On the other hand, for every $R\ge \max\{\ell,R_0\}$ we have that
$Q_{R-\ell}\subset Q_{R,\ell}\subset Q_{2\,R}$.
Thus by the comparison principle
\begin{equation}
\label{tri}
\lim_{R\to+\infty} w_{R,\ell}(x',x_N)=w_\Omega(x',x_N).
\end{equation}
Eventually, \eqref{uan}, \eqref{ciu} and \eqref{tri} imply the claim.
\vskip.2cm\noindent
{\bf Step 2.} Here we prove that 
\[
w_\Omega\in H^1(\omega\times (-R_0,R_0)),\qquad \mbox{ for every } R_0>0,
\]
which enforces the general result of Proposition \ref{lm:berry}. 
\par
We set as before $Q_R=\Omega\cap (-R,R)^N$ and call $w_R=w_{Q_R}$. We fix $R_0>0$ and consider $R>R_0+1$. We then take a one-dimensional cut-off function $\eta$ supported on $[-R_0-1,R_0+1]$ such that
\[
0\le \eta\le 1,\qquad \eta=1 \mbox{ on } [-R_0,R_0],\qquad \eta'\le 1.
\]
In the equation verified by $w_R$, we insert the test function 
\[
\varphi(x',x_N)=w_R(x',x_N)\,\eta^2(x_N).
\]
After some standard manipulations, we get
\[
\begin{split}
\int_{\omega\times (-R_0-1,R_0+1)} |\nabla w_R|^2\,\eta^2\,dx&\le C\,\int_{\omega\times (-R_0-1,R_0+1)} w_R\,\eta^2\,dx\\
&+C\,\int_{\omega\times (-R_0-1,R_0+1)} w_R^2\,|\eta'|^2\,dx.
\end{split}
\]
By recalling that $0\le w_R\le w_\Omega$ and that\footnote{The set $\Omega$ is bounded in every direction orthogonal to the $x_N$ axis, thus it is classical to see that $\lambda_1(\Omega)>0$. Then $w_\Omega\in L^\infty(\Omega)$ by Proposition \ref{prop:equivalenze}.} $w_\Omega\in L^\infty(\Omega)$, from the previous argument we get
\[
\int_{\omega\times (-R_0,R_0)} |\nabla w_R|^2\,dx\le C\,|\omega|\,R_0\,\|w_\Omega\|_{L^\infty(\Omega)}\,\Big(\|w_\Omega\|_{L^\infty(\Omega)}+1\Big),
\]
for every $R\gg 1$. This gives a uniform $H^1$ estimate on $\omega\times (-R_0,R_0)$ that we can take to the limit and obtain the desired Sobolev regularity of $w_\Omega$.
\begin{figure}[!h]
\includegraphics[scale=.7]{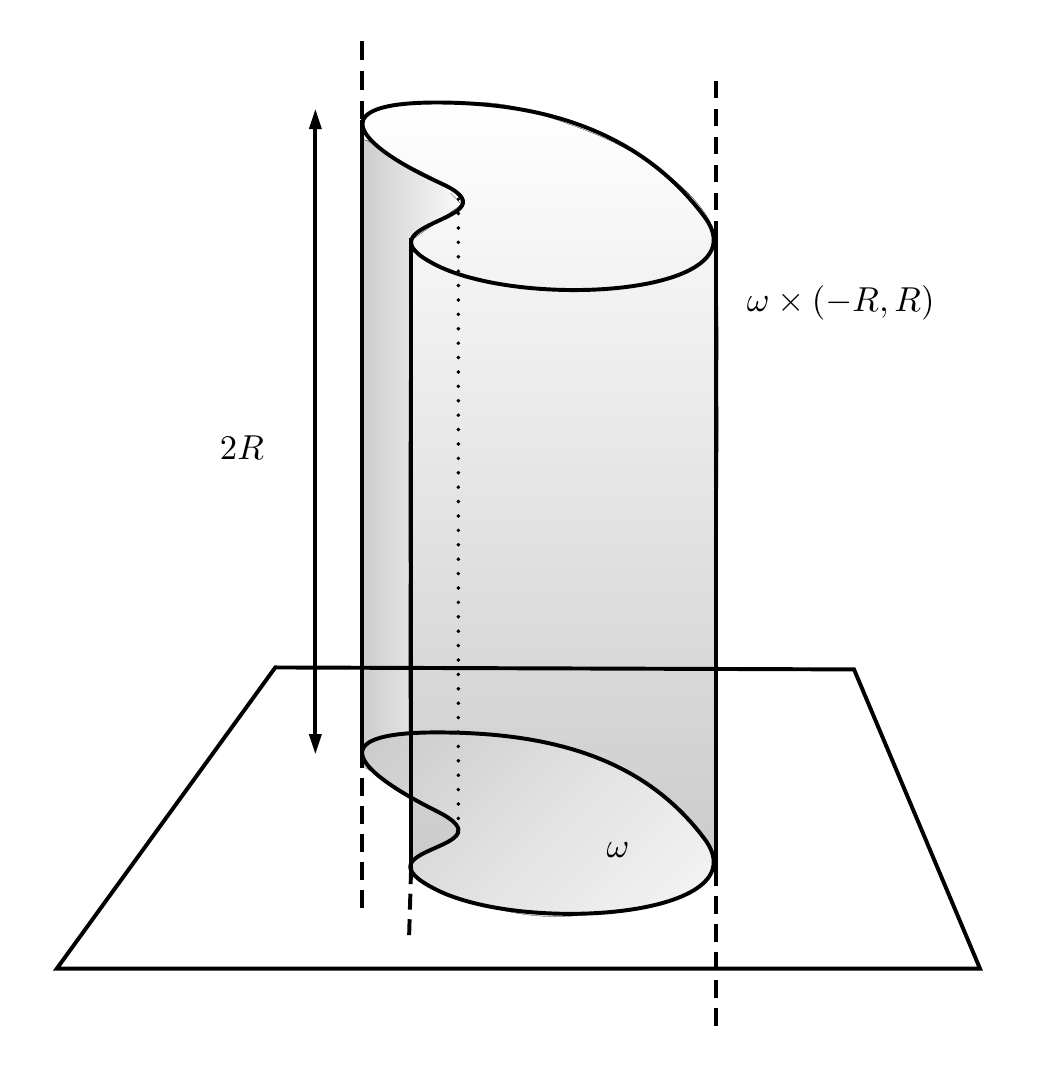}
\caption{A rectilinear wave-guide.}
\end{figure}

\vskip.2cm\noindent
{\bf Step 3.} We now prove that for every $R_0>0$, the torsion function $w_\Omega$ solves the mixed boundary value problem
\begin{equation}
\label{problemino}
\left\{\begin{array}{rcll}
-\Delta u&=& 1,& \mbox{ in } \omega\times (-R_0,R_0),\\
u&=&0, & \mbox{ on } \partial\omega\times (-R_0,R_0),\\
u_{x_N}&=&0,& \mbox{ on } \omega\times \{-R_0,R_0\}.
\end{array}
\right.
\end{equation}
We first observe that $w_\Omega$ is a solution of the equation in $\omega\times (-R_0,R_0)$. Indeed, it is sufficient to pass to the limit in the equation satisfied by $w_R$ and use the uniform $H^1$ estimate above.
\par
As for the boundary conditions, we observe that the Neumann one follows since $w_\Omega$ does not depend on the $x_N$ variable, by Step 1.
The compactness of the trace operator 
\[
H^1(\omega\times(- R_0,R_0))\hookrightarrow L^2(\partial(\omega\times(-R_0,R_0)),
\]
and the uniform $H^1$ estimate of Step 2 for $w_R$ 
imply the Dirichlet condition on the lateral boundary.
\vskip.2cm\noindent
{\bf Step 4.} In order to conclude, it is sufficient to observe that by Step 3 $w_\Omega$ and $w_\omega$ both solve \eqref{problemino}. Since the solution to the latter is unique, this gives the desired conclusion \eqref{dimreduct}.
\end{proof}
When the cross-section $\omega\subset\mathbb{R}^{N-1}$ of the wave-guide has a particular geometry, we can explicitely compute $w_\Omega$ and thus the limit potential \eqref{KJV}. For example, in the case that the cross-section is a $(N-1)-$dimensional ball, i.e. when
\[
\Omega=\{x'\in\mathbb{R}^{N-1}\, :\, |x'|<1\}\times \mathbb{R},
\]
then by Lemma \ref{exa:rectiWG} we have that
\[
w_\Omega(x',x_N)=\frac{1-|x'|^2}{2\,(N-1)},\qquad |x'|<1,
\]  
and thus
\[
\left|\frac{\nabla w_\Omega}{w_\Omega}\right|^2=-\frac{|x'|^2}{(1-|x'|^2)^2}.
\]
As before, we get that for every $V\in L^2_{\rm loc}(\Omega)$ such that
\[
0\ge V(x',x_N)\ge -\frac{|x'|^2}{(1-|x'|^2)^2},
\]
it holds that
\[
\lambda_1(\Omega;V)\ge \frac{1}{2\,\|w_\Omega\|_{L^\infty(\Omega)}}=N-1.
\]

\appendix

\section{A local $L^\infty$ estimate for Lane-Emden densities}

We recall that the volume of the unit ball in $\mathbb R^N$ is given by
\[
\omega_N=\frac{\pi^{N/2}}{\Gamma(N/2+1)},
\]
where $\Gamma$ is the usual Gamma function.
For $N\ge 3$, we denote by
\[
T_N=\sup_{u\in C^\infty_0(\mathbb{R}^N)}\left\{\left(\int_{\mathbb{R}^N} |u|^{2^*}\,dx\right)^\frac{2}{2^*}\, :\, \|\nabla u\|_{L^2(\mathbb{R}^N)}=1\right\}.
\]
the optimal constant in the Sobolev inequality for $\mathcal{D}^{1,2}_0(\mathbb R^N)$, i.e. the lowest number $C>0$ such that
\[
\left(\int_{\mathbb{R}^N} |u|^{2^*}\,dx\right)^\frac{2}{2^*}\le C\,\int_{\mathbb{R}^N} |\nabla u|^2\,dx,
\]
holds for any $u\in C^\infty_0(\mathbb{R}^N)$.
We recall that this is given by (see \cite{Ta})
\[
T_N=\pi\,N\,(N-2)\,\left(\frac{\Gamma(N/2)}{\Gamma(N)}\right)^\frac{2}{N}.
\]
In Section \ref{sec:4}, we needed a local $L^\infty$ estimate for weak subsolutions of the Lane-Emden equation. The proof is standard routine in Elliptic Regularity Theory, our main concern is in the explicit expression of the constant $\mathcal{C}$ appearing in the estimate. For this reason, we provide a detailed proof.
\begin{lm}
\label{lm:milano}
Let $\lambda>0$ and $1\le q<2$. Let $u\in H^1_{\rm loc}(\Omega)\cap L^\infty_{\rm loc}(\Omega)$ be a positive function such that
\[
\int \langle \nabla u,\nabla\varphi\rangle\,dx\le \lambda\,\int u^{q-1}\,\varphi\,dx,
\]
for every positive $\varphi\in H^1_0(B)$ and every ball $B\Subset\Omega$. Then for every ball $B_{R_0}\Subset\Omega$ and every $\alpha\ge 2$ we have
\begin{equation}
\label{brascoMUTO}
\|u\|_{L^\infty(B_{R_0/2})}\le \mathcal{C}\,\left[\left(\fint_{B_{R_0}} u^\alpha\,dx\right)^\frac{1}{\alpha}+\left(\frac{\lambda}{4}\right)^\frac{1}{2-q}\, R_0^\frac{2}{2-q}\right],
\end{equation}
where the constant $\mathcal{C}>0$ is given by
\[
\mathcal{C}=\left\{\begin{array}{ll}
\sqrt{\omega_N}\,\left(\dfrac{4\,N}{N-2}\right)^\frac{N\,(N-2)}{8}\,\Big(640\,T_N\Big)^\frac{N}{4},&\mbox{ for } N\ge 3,\\
&\\
\sqrt{\pi}\,(2\,\gamma)^{\frac{\gamma}{(\gamma-2)^2}}\,\left(\dfrac{640}{\lambda_{2,\gamma}(B_{1})}\right)^\frac{\gamma}{2\,(\gamma-2)},& \mbox{ for } N=2,\\
&\\
8\,\sqrt{5},& \mbox{ for } N=1.
\end{array}
\right.
\]
Here $\gamma$ is any number larger that $2$ and $\lambda_{2,\gamma}(B_1)$ is the Sobolev-Poincar\'e constant defined in \eqref{lambda2g}.
\end{lm}
\begin{proof}
We divide the proof in three cases, depending on the dimension $N$.
\vskip.2cm\noindent
{\bf Case $N\ge 3$.} We take $R_0/2\le r<R\le R_0$ and a pair of concentric balls $B_r\subset B_R\Subset \Omega$. We use as a test function
\[
\varphi=\eta^2\,(u+\delta)^\beta,
\]
where $\delta>0$, $\beta\ge 1$ and $\eta$ is a standard cut-off function, supported on $B_R$ and constantly $1$ on $B_r$, such that
\[
|\nabla \eta|\le \frac{1}{R-r}.
\] 
With standard manipulations, we obtain that
\[
\begin{split}
\int \left|\nabla (u+\delta)^\frac{\beta+1}{2}\right|^2\,\eta^2\,dx&\le \left(\frac{\beta+1}{\beta}\right)^2\,\int |\nabla \eta|^2\,(u+\delta)^{\beta+1}\,dx\\
&+\lambda\, \frac{2}{\beta}\,\left(\frac{\beta+1}{2}\right)^2\,\int \eta^2\,(u+\delta)^{\beta+q-1}\,dx.
\end{split}
\]
We now observe that
\[
(u+\delta)^{\beta+q-1}\le (u+\delta)^{\beta+1}\,\delta^{q-2}\qquad \mbox{ and }\qquad \left(\frac{\beta+1}{\beta}\right)^2\le \frac{4}{\beta}\,\left(\frac{\beta+1}{2}\right)^2,
\]
thus we get that
\[
\begin{split}
\int \left|\nabla (u+\delta)^\frac{\beta+1}{2}\right|^2\,\eta^2\,dx&\le 4\,\beta\,\left[\frac{1}{(R-r)^2}+\lambda\,\delta^{q-2}\right]\,\int_{B_R} (u+\delta)^{\beta+1}\,dx.
\end{split}
\]
We add on both sides the term
\[
\int |\nabla \eta|^2\,(u+\delta)^{\beta+1}\,dx,
\]
and we obtain that
\[
\begin{split}
\int \left|\nabla \left((u+\delta)^\frac{\beta+1}{2}\,\eta\right)\right|^2\,dx&\le 10\,\beta\,\left[\frac{1}{(R-r)^2}+\lambda\,\delta^{q-2}\right]\,\int_{B_R} (u+\delta)^{\beta+1}\,dx.
\end{split}
\]
We then use Sobolev inequality on the left-hand side, so to obtain
\begin{equation}
\label{N3}
\left(\int_{B_r} \left((u+\delta)^\frac{\beta+1}{2}\right)^{2^*}\,dx\right)^\frac{2}{2^*}\le 10\,T_N\,\beta\,\left[\frac{1}{(R-r)^2}+\lambda\,\delta^{q-2}\right]\,\int_{B_R} (u+\delta)^{\beta+1}\,dx.
\end{equation}
We now introduce the sequence of diverging exponents
\[
\vartheta_i=\frac{\beta_i+1}{2}=\left(\frac{2^*}{2}\right)^i,\qquad i\in\mathbb{N},
\]
and the sequence of shrinking radii
\[
R_i=r_0+\frac{R_0-r_0}{2^i},\qquad i\in\mathbb{N}.
\]
From \eqref{N3}, we get the iterative scheme 
\[
\begin{split}
\left(\int_{B_{R_{i+1}}} (u+\delta)^{2\,\vartheta_{i+1}}\,dx\right)^\frac{1}{2\,\vartheta_{i+1}}&\le \left(80\,T_N\,\left[\frac{4^i}{(R_0-r_0)^2}+\lambda\,\delta^{q-2}\right]\right)^\frac{1}{2\,\vartheta_{i}}\,(\vartheta_i)^\frac{1}{2\,\vartheta_i}\\
&\times\left(\int_{B_{R_i}} (u+\delta)^{2\,\vartheta_i}\,dx\right)^\frac{1}{2\,\vartheta_i}.
\end{split}
\]
Before launching the Moser's iteration, it is time to declare our choice of $\delta>0$: we take it to be
\begin{equation}
\label{delta}
\delta=(\lambda\,(R_0-r_0)^2)^\frac{1}{2-q}.
\end{equation}
Thus we get that
\[
\begin{split}
\left(\int_{B_{R_{i+1}}} (u+\delta)^{2\,\vartheta_{i+1}}\,dx\right)^\frac{1}{2\,\vartheta_{i+1}}&\le \left(\frac{160\,T_N}{(R_0-r_0)^2}\right)^\frac{1}{2\,\vartheta_{i}}\,(4^i\,\vartheta_i)^\frac{1}{2\,\vartheta_i}\\
&\times\left(\int_{B_{R_i}} (u+\delta)^{2\,\vartheta_i}\,dx\right)^\frac{1}{2\,\vartheta_i}.
\end{split}
\]
We start from $i=0$ and iterate infinitely many times. We end up with the estimate
\[
\|u+\delta\|_{L^\infty(B_{r_0})}\le C_N\,\frac{\Big(160\,T_N\Big)^\frac{N}{4}}{(R_0-r_0)^\frac{N}{2}}\,\left(\int_{B_{R_0}} (u+\delta)^2\,dx\right)^\frac{1}{2},
\]
with 
\[
C_N=\left(\frac{4\,N}{N-2}\right)^\frac{N\,(N-2)}{8}.
\] 
In particular, by taking $r_0=R_0/2$, we obtain with simple manipulations that
\[
\|u\|_{L^\infty(B_{R_0/2})}\le \sqrt{\omega_N}\,C_N\,\Big(640\,T_N\Big)^\frac{N}{4}\,\left[\left(\fint_{B_{R_0}} u^2\,dx\right)^\frac{1}{2}+\delta\right].
\]
We now recall the definition \eqref{delta} of $\delta$, thus the previous estimate rewrites as
\[
\begin{split}
\|u\|_{L^\infty(B_{R_0/2})}&\le \sqrt{\omega_N}\,C_N\,\Big(640\,T_N\Big)^\frac{N}{4}\\
&\times\left[\left(\fint_{B_{R_0}} u^2\,dx\right)^\frac{1}{2}+\left(\lambda\left(\frac{R_0}{2}\right)^2\right)^\frac{1}{2-q}\right].
\end{split}
\]
By Jensen's inequality, we can eventually replace the $L^2$ norm on the right-hand side by any $L^\alpha$ norm with $\alpha\ge 2$.
\vskip.2cm\noindent
{\bf Case $N=2$.} The proof runs as before, the only difference is that we now use Sobolev-Poincar\'e inequality for the embedding $\mathcal{D}^{1,2}_0(B_R)\hookrightarrow L^\gamma(B_R)$, in place of Sobolev inequality. Here $\gamma$ is any exponent larger than $2$. Thus, in place of \eqref{N3} we now get
\[
\left(\int_{B_r} \left((u+\delta)^\frac{\beta+1}{2}\right)^{\gamma}\,dx\right)^\frac{2}{\gamma}\le \frac{10\,\beta}{\lambda_{2,\gamma}(B_R)}\,\left[\frac{1}{(R-r)^2}+\lambda\,\delta^{q-2}\right]\,\int_{B_R} (u+\delta)^{\beta+1}\,dx.
\]
We used the notation
\[
\lambda_{2,\gamma}(B_R)=\min_{u\in\mathcal{D}^{1,2}_0(B_R)}\left\{\int_{B_R} |\nabla \varphi|^2\,dx\,\ :\, \|\varphi\|_{L^\gamma(B_R)}=1\right\}.
\]
Accordingly, we modify the definition of the exponents $\vartheta_i$ as follows 
\[
\vartheta_i=\frac{\beta_i+1}{2}=\left(\frac{\gamma}{2}\right)^i,\qquad i\in\mathbb{N},
\]
then we still take the sequence of shrinking radii
\[
R_i=r_0+\frac{R_0-r_0}{2^i},\qquad i\in\mathbb{N}.
\]
We get the iterative scheme
\[
\begin{split}
\left(\int_{B_{R_{i+1}}} (u+\delta)^{2\,\vartheta_{i+1}}\,dx\right)^\frac{1}{2\,\vartheta_{i+1}}&\le \left(\frac{80}{\lambda_{2,\gamma}(B_{R_0})}\,\left[\frac{4^i}{(R_0-r_0)^2}+\lambda\,\delta^{q-2}\right]\right)^\frac{1}{2\,\vartheta_{i}}\,(\vartheta_i)^\frac{1}{2\,\vartheta_i}\\
&\times\left(\int_{B_{R_i}} (u+\delta)^{2\,\vartheta_i}\,dx\right)^\frac{1}{2\,\vartheta_i},
\end{split}
\]
where we also used that
\[
\lambda_{2,\gamma}(B_{R_i})\ge \lambda_{2,\gamma}(B_{R_0}),\qquad \mbox{ for every }i\in\mathbb{N}.
\]
We still take $\delta$ as in \eqref{delta}. After infinitely many iterations, we now get
\[
\|u+\delta\|_{L^\infty(B_{r_0})}\le C_{\gamma}\,\frac{\Big(160\,\lambda_{2,\gamma}(B_{R_0})^{-1}\Big)^\frac{\gamma}{2\,(\gamma-2)}}{(R_0-r_0)^\frac{\gamma}{\gamma-2}}\,\left(\int_{B_{R_0}} (u+\delta)^2\,dx\right)^\frac{1}{2},
\]
with 
\[
C_{\gamma}=(2\,\gamma)^{\frac{\gamma}{(\gamma-2)^2}}.
\]
Finally, we observe that by scaling (recall that we are in dimension $N=2$)
\[
\lambda_{2,\gamma}(B_{R_0})=R_0^{-\frac{4}{\gamma}}\,\lambda_{2,\gamma}(B_1),
\]
thus by taking $r_0=R_0/2$ we obtain
\[
\|u\|_{L^\infty(B_{R_0/2})}\le\sqrt{\pi}\, C_{\gamma}\,\Big(640\,\lambda_{2,\gamma}(B_{1})^{-1}\Big)^\frac{\gamma}{2\,(\gamma-2)}\,\left(\fint_{B_{R_0}} (u+\delta)^2\,dx\right)^\frac{1}{2}.
\]
By recalling the definition of $\delta$, we get the conclusion.
\vskip.2cm\noindent
{\bf Case $N=1$.} This is the simplest case. We take the test function
\[
\varphi=\eta^2\,(u+\delta),
\]
where $\eta$ is a standard cut-off function as above, associated with a pair of concentric intervals of width $2\,r_0<2\,R_0$. For simplicity, we suppose them to be centered at the origin. By proceeding as before with $\beta=1$, we arrive at
\[
\begin{split}
\int_{-R_0}^{R_0} \left|\Big((u+\delta)\,\eta\Big)'\right|^2\,dx&\le 10\,\left[\frac{1}{(R_0-r_0)^2}+\lambda\,\delta^{q-2}\right]\,\int_{-R_0}^{R_0} (u+\delta)^2\,dx.
\end{split}
\]
We observe that by Sobolev embedding in dimension $1$ we have that
\[
\begin{split}
\int_{-R_0}^{R _0}\left|\left((u+\delta)\,\eta\right)'\right|^2\,dx&\ge \frac{1}{2\,R_0}\,\|(u+\delta)\,\eta\|^2_{L^\infty(-R_0,R_0)}\\
&\ge \frac{1}{2\,R_0}\,\|u\|^2_{L^\infty(-r_0,r_0)}.
\end{split}
\]
We still make the choice \eqref{delta} for $\delta$, then we get that
\[
\|u\|_{L^\infty(-r_0,r_0)}\le \frac{4\,\sqrt{5}\,\,R_0}{R_0-r_0}\, \left(\fint_{-R_0}^{R_0} (u+\delta)^2\,dx\right)^\frac{1}{2}.
\]
By using Minkowski inequality and recalling the definition of $\delta$, we conclude by taking $r_0=R_0/2$.
\end{proof}


\begin{thebibliography}{100}

\bibitem{BK} M. Belloni, B. Kawohl, A direct uniqueness proof for equations involving the $p-$Laplace operator, Manuscripta Math., {\bf 109} (2002), 229--231.

\bibitem{brafra} L. Brasco, G. Franzina, Convexity properties of Dirichlet integrals and Picone-type inequalities, Kodai Math. J., {\bf 37} (2014), 769--799. 

\bibitem{braruf} L. Brasco, B. Ruffini, Compact Sobolev embeddings and torsion functions, to appear on Ann. Inst. H. Poincar\'e Anal. Non Lin\'eaire, {\tt doi:10.1016/j.anihpc.2016.05.005}


\bibitem{vabu} 
M. van den Berg, D. Bucur, On the torsion function with Robin or Dirichlet boundary conditions, J. Funct. Anal., {\bf 266} (2014), 1647--1666.

\bibitem{vaca}  
M. van den Berg, T. Carroll, Hardy inequality and $L^p$ estimates for the torsion function, Bull. Lond. Math. Soc., {\bf 41} (2009), 980--986.

\bibitem{BO} H. Brezis, L. Oswald, Remarks on sublinear elliptic equations, Nonlinear Anal., {\bf 10} (1986), 55--64.

\bibitem{bubu} 
D. Bucur, G. Buttazzo, On the characterization of the compact embedding of Sobolev spaces, Calc. Var. Partial Differential Equations, {\bf 44} (2012), 455--475.

\bibitem{bubuve}
D. Bucur, G. Buttazzo, B. Velichkov, Spectral optimization problems for potentials and measures, Siam J. Math. Anal., {\bf 46} (2014), 2956--2986.  

\bibitem{DL} J. Deny, J.-L. Lions, Les espaces du type de Beppo Levi, Ann. Inst. Fourier, {\bf 5} (1954), 305--370.

\bibitem{devfrapin} B. Devyver, M. Fraas, Y. Pinchover, Optimal Hardy weight for second-order elliptic operator: an answer to a problem of Agmon, J. Funct. Anal., {\bf 266} (2014), 4422--4489.   

\bibitem{FraLa} G. Franzina, P. D. Lamberti, Existence and uniqueness for a $p-$Laplacian nonlinear eigenvalue problem, Electron. J. Differential Equations, {\bf 26} (2010), 10 pp. 

\bibitem{Lu} D. Lundholm, Geometric extensions of many-particle Hardy inequalities, J. Phys. A, {\bf 48} (2015), 175203, 25 pp.

\bibitem{maz} V. Maz'ya, Sobolev spaces, Sobolev spaces with applications to elliptic partial differential equations. Second, revised and augmented edition. Grundlehren der Mathematischen Wissenschaften [Fundamental Principles of Mathematical Sciences], {\bf 342}. Springer, Heidelberg, 2011. 

\bibitem{moser} J. Moser, On Harnack's Theorem for Elliptic Differential Equations, Comm. Pure Appl. Math., {\bf 16} (1961), 577--591.

\bibitem{OK} B. Opic, A. Kufner, Hardy-type inequalities. Pitman Research Notes in Mathematics Series, {\bf 219}. Longman Scientific \& Technical, Harlow, 1990.

\bibitem{Ta} G. Talenti, Best constant in Sobolev inequality, Ann. Mat. Pura Appl., {\bf 110} (1976), 353--372.

\bibitem{Te} G. Teschl, Mathematical methods in quantum mechanics. With applications to Schr\"odinger operators. Second edition. Graduate Studies in Mathematics, {\bf 157}. American Mathematical Society, Providence, RI, 2014.

\end{thebibliography}
\end{document}